\documentclass[12pt]{amsart}
\usepackage{amsaddr}
\usepackage{tikz} \usepackage{tikz-cd}
\usepackage{amssymb}
\usepackage{amsmath}
\usepackage{hyperref}
\usepackage{comment}
\hypersetup{
    breaklinks=true,
    colorlinks=true,
    citecolor=red,
    linkcolor=blue,
    filecolor=magenta,      
    urlcolor=cyan,
}

\usepackage{collcell}
\newcommand{\shiftdown}[1]{\smash{\raisebox{-.5\normalbaselineskip}{#1}}}
\newcolumntype{C}{>{\collectcell\shiftdown}c<{\endcollectcell}}

\DeclareSymbolFont{extraup}{U}{zavm}{m}{n}
\DeclareMathSymbol{\varheart}{\mathalpha}{extraup}{86}
\DeclareMathSymbol{\vardiamond}{\mathalpha}{extraup}{87}

\DeclareMathOperator\C{\mathbb C}
\DeclareMathOperator\N{\mathbb N}
\DeclareMathOperator{\Graph}{Graph}
\DeclareMathOperator\h{\hbar}
\DeclareMathOperator\Hom{Hom}
\DeclareMathOperator\End{End}
\DeclareMathOperator\Fl{\text{Fl}}
\DeclareMathOperator\GL{\text{GL}}
\DeclareMathOperator\Mat{\text{Mat}}
\DeclareMathOperator\Or{\mathcal{O}}
\DeclareMathOperator\MM{\mathbb{M}}
\DeclareMathOperator\NN{\mathbb{N}}

\def\r{r}
\def\c{c}
\DeclareMathOperator\BCT{BCT}
\DeclareMathOperator{\Pe}{\mathbb{P}}
\def\leqb{\leq_B}
\def\leqsb{\leq_{\hat{B}}}
\def\geqsb{\geq_{\hat{B}}}
\def\leqgb{\leq_{G}}
\newcommand\bs{{\color{blue} \char`\\}}
\newcommand\fs{{\color{red}/}}
\def\ttt#1{\text{\tt #1}}
\DeclareMathOperator\Zb{\mathcal Z}
\newcommand{\Ab}{\mathcal A}
\newcommand{\Xb}{\mathcal X}
\DeclareMathOperator\At{\mathbb A}
\DeclareMathOperator\Tt{\mathbb T}
\def\NS5{N\!S5}
\def\Ch{X}
\def\DD{\mathcal D}
\def\ch{ch}

\usepackage{mathabx,epsfig}

\newtheorem{fact}{Fact}[section]
\newtheorem{lemma}[fact]{Lemma}
\newtheorem{theorem}[fact]{Theorem}
\newtheorem{thmx}{Theorem}
\newtheorem{definition}[fact]{Definition}
\newtheorem{example}[fact]{Example}
\newtheorem{ccounterexample}[fact]{Counterexample}
\newtheorem{rremark}[fact]{Remark}
\newenvironment{remark}{\begin{rremark} \rm}{\end{rremark}}
\newenvironment{counterexample}{\begin{ccounterexample} \rm}{\end{ccounterexample}}

\newtheorem{corollary}[fact]{Corollary}

\DeclareMathOperator\Attr{Att}
\newcommand{\Att}[1]{\Attr_{\mathfrak{#1}}}
\def\Cs{\C^{\times}}
\DeclareMathOperator\TsFl{T^*Fl}
\DeclareMathOperator\TsP{T^*\Pe}
\DeclareMathOperator\TsGr{T^*Gr}
\def\X{X}

\title[]{Geometric Bruhat order on (0,1)-matrices} %with given row and column sum}

\author{T. M. Botta$^\diamond$, A. Foster$^\star$, R. Rim\'anyi$^\star$}

\email{tommaso.botta@math.ethz.ch}

\email{aofoster@live.unc.edu}

\email{rimanyi@email.unc.edu}

\address{$^\diamond$ Department of Mathematics, ETH Z\"urich, Switzerland \\
$^\star$ Department of Mathematics, University of North Carolina at Chapel Hill, USA}

\begin{document}

\maketitle

\begin{abstract} 
The combinatorially and the geometrically defined partial orders on the set of permutations coincide. We extend this result to $(0,1)$-matrices with fixed row and column sums. Namely, the Bruhat order induced by the geometry of a Cherkis bow variety of type A coincides with one of the two combinatorially defined Bruhat orders on the same set. 
\end{abstract}

\section{Introduction}

Let $\BCT(r,c)$ denote the set of 01 matrices with row sum vector $r$ and column sum vector $c$---BCT stands for binary contingency table. Recent advances in enumerative geometry identified a class of smooth varieties $\X(r,c)$, called Cherkis bow varieties of type A (``bow varieties''), that are parameterized by pairs of vectors $r,c$. Bow varieties are endowed with a torus action, and the torus fixed points of $\X(r,c)$ are naturally identified with $\BCT(r,c)$. Bow varieties are very general (eg. they contain Nakajima quiver varieties of type A as special cases), they have significance in physics \cite{BFN}, and they have favorable properties for enumerative geometry \cite{RimanyiShou, BottaRimanyi} (eg. bow varieties are closed for ``combinatorial 3d mirror symmetry'').  

For $r=c=(1,1,\ldots,1)\in \N^n$ the bow variety is the total space of the cotangent bundle over the full flag variety $\Fl(n)$, and the torus action is the natural $(\Cs)^n$ action. The fixed point set is $\BCT((1,\ldots,1),(1,\ldots,1))$ which is the set of $n\times n$ permutation matrices. More generally, if only $c=(1,1,\ldots,1)$ then the bow variety is the total space of the cotangent bundle of a partial flag variety on $\C^n$, with its natural $(\Cs)^n$ action. The fixed point set is $\BCT(r,(1,\ldots,1))$ which is indeed in obvious bijection with $S_n/(\times_i S_{r_i})$.

In these Schubert Calculus examples one can use geometry to define a partial order, called {\em geometric Bruhat order}, on the set of fixed points. Namely, after fixing a reference flag the variety is partitioned into Schubert cells, with each Schubert cell containing one torus fixed point (we can forget the cotangent directions for now). The ``containment in the closure'' relation then defines a partial order on the set of Schubert cells, and hence also on the set of fixed points. 

It turns out that there is an analogous partial order on the torus fixed points of general bow varieties $\X(r,c)$. The role of the reference flag is played by a general cocharacter of the torus. The role of Schubert cells is played by attracting sets with respect to this cocharacter. In fact, one needs to work with iterated attracting sets (a.k.a. full attracting sets). The ``containment in the closure'' part of the definition is also modified: we consider which fixed points belong to the closure of which attracting sets. Nevertheless, there is a notion of geometric Bruhat order on $\BCT(r,c)$ for general $r,c$ vectors that generalizes the notion of geometric Bruhat order in Schubert Calculus. This Bruhat order plays an important role in the theory of Stable Envelopes developed by Okounkov and his coauthors \cite{maulik2012quantum, Okounkov_lectures, aganagic2016elliptic}.

\smallskip

Our paper is about finding a combinatorial characterization of the geometric Bruhat order on $\BCT(r,c)$. In fact, combinatorists already studied ``natural'' partial orders on 01 matrices with prescribed row and column sums \cite{BrualdiBook}. The surprising outcome of their studies is that there are two different combinatorially natural partial orders, that they named Bruhat order and secondary Bruhat order. We will call them the {\em combinatorial Bruhat order} and the {\em combinatorial secondary Bruhat order}. The two orders coincide if $c=(1,1,\ldots,1)$ and it is a well-known fact of Schubert Calculus that in this case the geometric Bruhat order is also the same. The main result of our paper states:

\begin{thmx}
    The geometric Bruhat order on $\BCT(r,c)$ is the same as the combinatorial secondary Bruhat order. 
\end{thmx}

After introducing the necessary combinatorics and geometry, we first define the geometric Bruhat order using the geometry of $\X(r,c)$. 
Then we present a characterization of the partial order using invariant curves on $\X(r,c)$. In Sections~\ref{sec:proof_invariant_curves} and \ref{sec:proof_resolution} we give two proofs of our main theorem. The first one relies on the recent work \cite{FosterShou} describing all invariant curves on $\X(r,c)$, and the second one relies on a technique, called {\em geometric resolution}, developed in \cite{BottaRimanyi} to prove the so-called 3d mirror symmetry property of bow varieties.

\smallskip

We will use the notation $\N=\{0,1,2,\ldots\}$ and $k^n=(k,k,\ldots,k)\in \N^n$.

\medskip

\noindent{\bf Acknowledgements.} The first author was supported as a part of NCCR SwissMAP, a National Centre of Competence in Research, funded by the Swiss National Science Foundation (grant number 205607) and grant 200021\_196892 of the Swiss National Science Foundation. The third author was supported by the NSF grants 2152309 and 2200867. We are grateful for useful discussions on the subject with R. Brualdi, A.~Buch, G. Felder, and Y.~Shou.

\section{Combinatorial Bruhat orders}

Let $\r\in \N^m$ and $\c\in \N^n$ satisfy $\sum_{i=1}^m r_i=\sum_{j=1}^n c_j$. Define $\BCT(r,c)$ to be the set of $m\times n$ matrices whose entries are all 0 or 1, and whose row sum vector is~$\r$ and column sum vector is~$\c$. %BCT stands for Binary Contingency Table.

\begin{example} \label{ex:c=1}\rm
The set of $n\times n$ permutation matrices is a special case for $\r=\c=(1,1,\ldots,1)$. A more general special case will play a role in this paper: let $\r=(r_1,r_2,\ldots,r_m)$, 
%$n=\sum_{i=1}^m r_i$, 
$n=\sum r_i$,
$\c=(1^n)$. In Algebra this $\BCT(\r,\c)$ is in bijection with $S_n/(\times_i S_{r_i})$, and in Geometry it parametrizes the torus fixed points (or the Schubert cells) of the partial flag variety 
%\{V_{\bullet} : V_1^{r_1} \subset V_2^{r_1+r_2} \subset \ldots \subset V_{m-1}^{\sum_{i=1}^{m-1}r_i} \subset \C^n\}
$
\{V_{\bullet} : V_1^{r_1} \subset V_2^{r_1+r_2} \subset \ldots \subset V_{m-1}^{r_1+r_2+\ldots+r_{m-1}} \subset \C^n\}
$.
\end{example}

%Elements of this $\BCT(\r,\c)$ are in an obvious bijection with %the $n!/\prod r_i!$ possible ways of writing $\{1,\ldots,n\}$ as a disjoint union $R_1 \sqcup R_2 \sqcup \ldots \sqcup R_m$ with $|R_i|=r_i$. In algebra this $\BCT(r,c)$ is $S_n/(\times_i S_{r_i})$, and in geometry it parametrizes the torus fixed points (or the Schubert cells) of the partial flag variety 
%\{V_{\bullet} : V_1^{r_1} \subset V_2^{r_1+r_2} \subset \ldots \subset V_{m-1}^{\sum_{i=1}^{m-1}r_i} \subset \C^n\}
%$
%\{V_{\bullet} : V_1^{r_1} \subset V_2^{r_1+r_2} \subset \ldots \subset %V_{m-1}^{r_1+r_2+\ldots+r_{m-1}} \subset \C^n\}
%$.
%\end{example}

Combinatorial properties of $\BCT(\r,\c)$ have been extensively studied, see \cite{BrualdiBook} for a survey. For example, the Gale-Ryser theorem \cite[Ch.~6, Thm.~1.1]{RyserBook} provides a simple necessary and sufficient condition for $\BCT(r,c)$ to be non-empty. Another fundamental result is that the graph whose vertex set is $\BCT(r,c)$ and whose edges are pairs of matrices that only differ in the intersection of two rows and two columns by 
\begin{equation}\label{eq:LI}
L_2=\begin{pmatrix} 0 & 1 \\ 1 & 0 \end{pmatrix} \leftrightarrow \begin{pmatrix} 1 & 0 \\ 0 & 1 \end{pmatrix}=I_2
\end{equation} 
is connected  \cite{Ryser1957} (the rows and columns are not necessarily adjacent).
In this paper we will be concerned with partial orders on $\BCT(\r,\c)$. 

\subsection{The combinatorial Bruhat order}

For $M\in \BCT(\r,\c)$ let $\Sigma_M=(\sigma(M)_{i,j})$ be the $m\times n$ matrix where 
\[  
\sigma(M)_{i,j}=\sum_{k=1}^i \sum_{l=1}^j m_{kl}.
\]
Define the {\em combinatorial Bruhat order} by $M_1\leqb M_2$ if $\Sigma_{M_1}\geq \Sigma_{M_2}$ entrywise. 

\subsection{The combinatorial secondary Bruhat order}
The {\em combinatorial secondary Bruhat order} is defined by letting $M_1\leqsb M_2$ if $M_2$ can be transformed into $M_1$ by a sequence of $L_2 \to I_2$ interchanges, see \eqref{eq:LI}. Nota bene, we permit the surgery depicted in \eqref{eq:LI} in the intersection of two (not necessarily consecutive) rows and two (not necessarily consecutive) columns {\em only in one direction}: from $L_2$ to $I_2$. 

Recall that the covering relation of a (finite) partial order $\leq$ is defined by: $M$ covers $M'$ provided $M' < M$ and there is no $M''$ with   $M' <  M'' < M$. A combinatorial description for the covering relation of $\leqsb$ is proved by Brualdi and Deaett.

\begin{theorem}\label{thm:cover}\cite[Thm.~3.1]{BrualdiDeaett}
Let $M=(m_{ij})$ be a matrix in $\BCT(r,c)$ where $M[\{i,j\},\{k,l\}]=L_2$ (cf. \eqref{eq:LI}). Let $M'$ be the matrix obtained from $M$ by the $L_2\to I_2$ interchange that replaces $M[\{i,j\},\{k,l\}]$ with $I_2$. Then $M$ covers $M'$ in the combinatorial secondary Bruhat order on $\BCT(\r,\c)$ if and only if
\begin{enumerate}
    \item $m_{pk}=m_{pl}$ for $i<p<j$,
    \item $m_{iq}=m_{jq}$ for $k<q<l$,
    \item $m_{pk}=0$ and $m_{iq}=0$ imply $m_{pq}=0$ for $i<p<j,k<q<l$,
    \item $m_{pk}=1$ and $m_{iq}=1$ imply $m_{pq}=1$ for $i<p<j,k<q<l$.
\end{enumerate}
\end{theorem}

\begin{example} \rm For $\r=\c=(2,1,2)$ the combinatorial Bruhat order and the combinatorial secondary Bruhat order on $\BCT(\r,\c)$ are the same. The Hasse diagram (only drawing the cover relations) is on the left of Figure~\ref{fig:212}.   
\begin{figure}
\begin{tikzpicture}[scale=0.8]
\node[] at (5,0) { $\footnotesize \begin{pmatrix} 1&1&0\\0&0&1\\1&0&1 \end{pmatrix}$};
\node[] at (0,0) {$\footnotesize  \begin{pmatrix} 1&0&1\\1&0&0\\0&1&1 \end{pmatrix}$};
\node[] at (2.5,2.5) {$\footnotesize\begin{pmatrix} 1&0&1\\0&1&0\\1&0&1 \end{pmatrix}$};
\node[] at (5,5) {$\footnotesize\begin{pmatrix} 0&1&1\\1&0&0\\1&0&1 \end{pmatrix}$};
\node[] at (0,5) {$\footnotesize\begin{pmatrix} 1&0&1\\0&0&1\\1&1&0 \end{pmatrix}$};
\draw [thick,blue] (.8,.8)--(1.7,1.7); \draw [thick,blue] (4.2,.8)--(3.3,1.7); 
\draw [thick,blue] (0.8,4.2)--(1.7,3.3);\draw [thick,blue] (4.2,4.2)--(3.3,3.3);
\node[blue,rotate=45] at (1.15,1.5) {$<$};
\node[blue,rotate=135] at (3.85,1.5) {$<$};
\node[blue,rotate=45] at (3.85,3.5) {$<$};
\node[blue,rotate=135] at (1.15,3.5) {$<$};
\end{tikzpicture}
\qquad\qquad
\begin{tikzpicture}[baseline=-17,scale=0.9]
\draw [thick,blue] (0,0)--(2.5,2.5); \draw [thick,blue] (5,0)--(2.5,2.5); 
\draw [thick,blue] (0,5)--(2.5,2.5);\draw [thick,blue] (5,5)--(2.5,2.5);
\draw [thick,blue] (0,0)--(0,5); \draw [thick,blue] (5,0)--(5,5);

\draw [thick,blue] (0,0)--(-1.5,1.5); \draw [thick,blue] (0,0)--(1.5,-1.5);
\draw [thick,blue] (0,0)--(-1.5,-1.5);\draw [thick,blue] (0,0)--(0,-1.5);
\draw [thick,blue] (0,5)--(-1.5,6.5);\draw [thick,blue] (0,5)--(0,6.5);
\draw [thick,blue] (0,5)--(1.5,6.5); \draw [thick,blue] (0,5)--(-1.5,3.5);

\draw [thick,blue] (5,0)--(6.5,1.5); \draw [thick,blue] (5,0)--(6.5,-1.5);
\draw [thick,blue] (5,0)--(3.5,-1.5);\draw [thick,blue] (5,0)--(5,-1.5);
\draw [thick,blue] (5,5)--(3.5,6.5);\draw [thick,blue] (5,5)--(5,6.5);
\draw [thick,blue] (5,5)--(6.5,6.5); \draw [thick,blue] (5,5)--(6.5,3.5);

\draw [thick,blue] (2.5,2.5)--(3.7,1); 
\draw [thick,blue] (2.5,2.5)--(3.4,1); 
\draw [thick,blue] (2.5,2.5)--(3 ,1); 
\draw [thick,blue] (2.5,2.5)--(2.7,1); 

\draw [thick,blue] (2.5,2.5)--(1.3,4); 
\draw [thick,blue] (2.5,2.5)--(1.6,4); 
\draw [thick,blue] (2.5,2.5)--(1.9,4); 
\draw [thick,blue] (2.5,2.5)--(2.3,4); 

% around f_1
\node at (-1.7,1) {\footnotesize $\frac{a_2}{a_3}\hbar$};
\node at (-.3,1) {\footnotesize $\frac{a_1}{a_3}$};
\node at (1.2,.4) {\footnotesize  {$\frac{a_1}{a_2}\h^{-1}$}};
\node at (-1.7,-1.1) {\footnotesize {$\frac{a_2}{a_1} \hbar^2$}};
\node at (-.4,-1.1) {\footnotesize {$\frac{a_3}{a_1}\hbar $}};
\node at (.6,-1.1) {\footnotesize  {$\frac{a_3}{a_2}$ }};

% around f_2
\node at (4.2,.3) {\footnotesize $\frac{a_2}{a_3}$};
\node at (5.3,1) {\footnotesize $\frac{a_1}{a_3}$};
\node at (6.6,1) {\footnotesize  {$\frac{a_2}{a_3}\hbar$}};
\node at (6.7,-1.1) {\footnotesize {$\frac{a_3}{a_2} \hbar $}};
\node at (5.4,-1.1) {\footnotesize {$\frac{a_3}{a_1} \hbar$}};
\node at (4.4,-1.1) {\footnotesize  {$\frac{a_3}{a_2} $}};

% around f_4
\node at (-1.7,4) {\footnotesize $\frac{a_3}{a_2}$};
\node at (-.3,4) {\footnotesize $\frac{a_3}{a_1}$};
\node at (1.2,4.6) {\footnotesize  {$\frac{a_3}{a_2}\hbar^{-1}$}};
\node at (-1.8,6.1) {\footnotesize {$\frac{a_2}{a_3} \hbar^2$}};
\node at (-.4,6.1) {\footnotesize {$\frac{a_1}{a_3} \hbar$}};
\node at (.6,6.1) {\footnotesize  {$\frac{a_2}{a_3} \hbar$}};

% around f_5
\node at (4.1,4.6) {\footnotesize $\frac{a_2}{a_1}$};
\node at (5.3,4) {\footnotesize $\frac{a_3}{a_1}$};
\node at (6.5,4) {\footnotesize  {$\frac{a_3}{a_2}$}};
\node at (6.7,6.1) {\footnotesize {$\frac{a_1}{a_2}\hbar$}};
\node at (5.4,6.1) {\footnotesize {$\frac{a_1}{a_3} \hbar$}};
\node at (4.4,6.1) {\footnotesize  {$\frac{a_2}{a_3} \hbar$}};

% around f_3
\node at (1.4,3) {\footnotesize $\frac{a_2}{a_3}\hbar$};
\node at (2.7,4) {\footnotesize $\frac{a_2}{a_3}\hbar$};
\node at (3.5,3) {\footnotesize  {$\frac{a_1}{a_2}$}};
\node at (1.5,2) {\footnotesize {$\frac{a_2}{a_1}\hbar$}};
\node at (2.4,1) {\footnotesize {$\frac{a_3}{a_2} $}};
\node at (3.5,2) {\footnotesize  {$\frac{a_3}{a_2} $}};

\node[fill,white,draw,circle,minimum size=.5cm,inner sep=0pt] at (0,0) {\ };
\node[draw,circle,minimum size=.5cm,inner sep=0pt] at (0,0) {\footnotesize $f_1$};
\node[fill,white,draw,circle,minimum size=.5cm,inner sep=0pt] at (5,0) {\ };
\node[draw,circle,minimum size=.5cm,inner sep=0pt] at (5,0) {\footnotesize $f_2$};
\node[fill,white,draw,circle,minimum size=.5cm,inner sep=0pt] at (2.5,2.5) {\ };
\node[draw,circle,minimum size=.5cm,inner sep=0pt] at (2.5,2.5) {\footnotesize $f_3$};
\node[fill,white,draw,circle,minimum size=.5cm,inner sep=0pt] at (0,5) {\ };
\node[draw,circle,minimum size=.5cm,inner sep=0pt] at (0,5) {\footnotesize $f_4$};
\node[fill,white,draw,circle,minimum size=.5cm,inner sep=0pt] at (5,5) {\ };
\node[draw,circle,minimum size=.5cm,inner sep=0pt] at (5,5) {\footnotesize $f_5$};
\end{tikzpicture}
\caption{{\em Left}: The Hasse diagram of $\BCT((2,1,2),(2,1,2))$ for either of the two combinatorial Bruhat orders $\leqb$, $\leqsb$. Observe that the two leftmost matrices (as well as the two rightmost matrices) are connected by an interchange \eqref{eq:LI}, but that relation is not a cover relation. {\em Right}: The fixed points and the invariant curves (together with their weights) of the bow variety $\Ch($\ttt{\fs 2\fs3\fs5\bs3\bs2\bs}$)$. The combs of blue lines around the middle indicate two 1-parameter families of invariant curves. The geometric Bruhat order, that is, the ``$a_{large}/a_{small}$ flow'' in the right picture coincides with the $>$ relation in the left picture. The bijective correspondence between interchanges of type \eqref{eq:LI} in the left picture and the ``bounded'' curves in the right picture holds for this example, but {not} in general.}
\label{fig:212}
\end{figure}
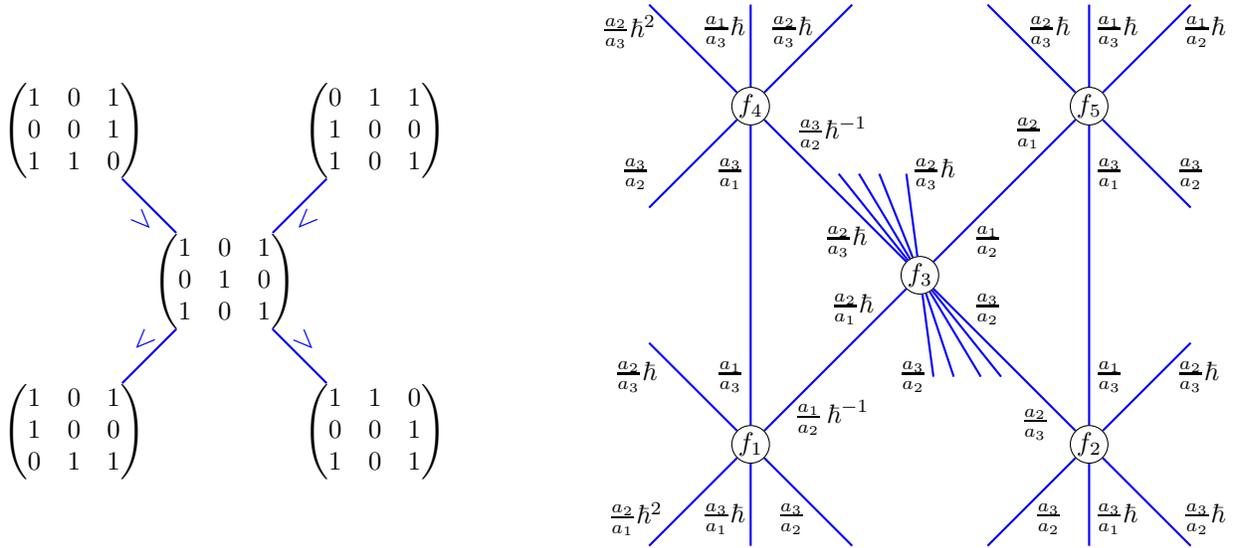
\end{example}

\subsection{The two combinatorial Bruhat orders are different}
\label{sect:The two combinatorial Bruhat orders are different}

It is clear that $M_1\leqsb M_2$ implies $M_1\leqb M_2$. However, the two partial orders are not the same. The counterexample from \cite[Sect.~2]{BrualdiDeaett} is the pair of matrices
\[
M_1=
\begin{pmatrix}
1&0&0&0&0&0\\
1&0&1&1&1&0\\
1&1&1&1&1&0\\
0&0&0&1&1&0\\
0&0&0&1&0&0\\
0&0&0&1&1&1
\end{pmatrix},
\qquad
M_2=
\begin{pmatrix}
0&0&0&1&0&0\\
1&1&0&1&1&0\\
1&0&1&1&1&1\\
1&0&0&1&0&0\\
0&0&0&0&1&0\\
0&0&1&1&1&0
\end{pmatrix},
\]
for which $M_1\leqb M_2$ but $M_1 \not\leqsb M_2$. This pair is the smallest counterexample the authors are aware of. They live in $\BCT((1,4,5,2,1,3),(3,1,2,5,4,1))$ which is an 89-element set. 

A well known fact is that for $r=c=1^n$, that is, for permutation matrices, the two combinatorial Bruhat orders coincide. However, a complete characterization of pairs of vectors $r,c$ for which the two combinatorial Bruhat orders coincide is not known. For some examples and counterexamples see \cite{Fernandes2019ClassesO, Zhang2022TheCO}. For example, it is proved that the orders $\leqb$ and $\leqsb$ coincide for $r=c=(k^n)$ if and only if either $0\leq n\leq 5$ or $k\in\{0,1,2,n-2,n-1,n\}$ with $n\geq 6$. Note that $|\BCT(3^6,3^6)|=297200$. It is also proved that the orders $\leqb$ and $\leqsb$ coincide if either $r$ or $c$ equals either $(1^n)$ or $(2^n)$. In particular, the two orders coincide for the set corresponding to the partial flag variety of Example \ref{ex:c=1}.

\section{Bow varieties}

\subsection{Brane Diagrams} \label{sec:branediagrams}

Combinatorial objects like
\begin{equation}
    \label{eq: example brane diagram}
\begin{tikzpicture}[baseline=0,scale=.4]
\begin{scope}[yshift=0cm]
\draw [thick,red] (0.5,0) --(1.5,2); 
\draw[thick] (1,1) node[left] {$\DD=$}--(2.5,1) node [above] {$2$} -- (31,1);
\draw [thick,blue](4.5,0) --(3.5,2);  
\draw [thick](4.5,1)--(5.5,1) node [above] {$2$} -- (6.5,1);
\draw [thick,red](6.5,0) -- (7.5,2);  %%
\draw [thick](7.5,1) --(8.5,1) node [above] {$2$} -- (9.5,1); 
\draw[thick,blue] (10.5,0) -- (9.5,2);  
\draw[thick] (10.5,1) --(11.5,1) node [above] {$4$} -- (12.5,1); 
\draw [thick,red](12.5,0) -- (13.5,2);   %%
\draw [thick](13.5,1) --(14.5,1) node [above] {$3$} -- (15.5,1);
\draw[thick,red] (15.5,0) -- (16.5,2);  %%
\draw [thick](16.5,1) --(17.5,1) node [above] {$3$} -- (18.5,1);  
\draw [thick,red](18.5,0) -- (19.5,2);  %%
\draw [thick](19.5,1) --(20.5,1) node [above] {$4$} -- (21.5,1);
\draw [thick,blue](22.5,0) -- (21.5,2);
\draw [thick](22.5,1) --(23.5,1) node [above] {$3$} -- (24.5,1);  
\draw[thick,red] (24.5,0) -- (25.5,2); 
\draw[thick] (25.5,1) --(26.5,1) node [above] {$2$} -- (27.5,1);
\draw [thick,blue](28.5,0) -- (27.5,2);  %%
\draw [thick](28.5,1) --(29.5,1) node [above] {$2$} -- (30.5,1);   
\draw [thick,blue](31.5,0) -- (30.5,2);   %%
\end{scope}
\end{tikzpicture}
\end{equation}
 will be called (type A) brane diagrams. The red forward-leaning lines are called NS5 branes, denoted by $\Zb$. The blue backward-leaning lines are called D5 branes, denoted by $\Ab$. The black lines between 5-branes are called D3 branes, denoted by $\Xb$, and the integer sitting there is called its multiplicity or dimension $d_{\Xb} >0$. For brevity, we often use the more compact notation
\[
\DD=\ttt{\fs $2$\bs $2$\fs $2$\bs $4$\fs $3$\fs $3$\fs $4$\bs $3$\fs $2$\bs $2$\bs }.
\]

If all NS5 branes are to the left of all D5 branes, we call the diagram {\em separated}, if all NS5 branes are to the right of all D5 branes, we call the diagram {\em co-separated}. An example of separated brane diagram is $\ttt{\fs $1$\fs$3$\fs $4$\fs $5$\bs $3$\bs $2$\bs}$.

The {\em charge} of an NS5  brane $\Zb$ or a D5 brane $\Ab$ is defined by
\begin{align*}
\ch(\Zb)= & (d_{\Zb^+}-d_{\Zb^-})+|\{\text{D5 branes left of $\Zb$}\}|,
\\ 
\ch(\Ab)= & (d_{\Ab^-}-d_{\Ab^+})+|\{\text{NS5 branes right of $\Ab$}\}|.
\end{align*}
Here, the superscripts $+, -$ refer to the D3 branes directly to the right and left. 
%We define the {\em local charge} (or ``{\em weight}'') of 5-branes by  $\w(\Zb)=|d_{\Zb^+}-d_{\Zb^-}|$, $\w(\Ab)=|d_{\Ab^+}-d_{\Ab^-}|$. 
%For an NS5 brane let $\ell(\Zb)$ denote the number of D5 branes left of $\Zb$. %(the number that comes up in the definition of its charge).
By convention, we order the D5 branes and the NS5 branes from left to right. For example, the leftmost NS5 brane in \eqref{eq: example brane diagram} is $\Zb_1$ and the rightmost is $\Zb_6$. We also collect the charges the D5 and the NS5 charges in vectors $c=(c_1, \dots c_n)\in \N^n$ and $r=(r_1,\dots r_m)\in \N^m$ by setting $c_i=\ch(\Ab_i)$ and $r_i=\ch(\Zb_i)$. We will refer to $c$ and $r$ as D5 and NS5 charge vector, respectively.

\subsection{Bow varieties} \label{sec:def of bow variety}
In this section we recall the quiver description of bow variety introduced by Nakajima and Takajiama in \cite{Nakajima_Takayama}. We refer to [loc. cit.] for more details. Let $\DD$ be a given brane diagram. To a D3 brane $\Xb$ we associate a complex vector space $W_{\Xb}$ of dimension $d_{\chi}$. To a D5 brane $\Ab$ we associate a one-dimensional space $\C_{\Ab}$ with the standard $\GL(\C_{\Ab})$ action and the ``three-way part''
\begin{align}
\begin{split}
\label{eq: three way part}
\MM_{\Ab}= &\Hom(W_{{\Ab}^+},W_{{\Ab}^-}) \oplus
\h\Hom(W_{{\Ab}^+},\C_{\Ab}) \oplus \Hom(\C_{\Ab},W_{{\Ab}^-}) \\
&  \oplus\h\End(W_{{\Ab}^-}) \oplus \h\End(W_{{\Ab}^+}),
\end{split}
\end{align} 
with elements denoted $(A_{\Ab}, b_{\Ab}, a_{\Ab}, B_{\Ab}, B'_{\Ab})$, and $\NN_{\Ab}=\h\Hom(W_{{\Ab}^+},W_{{\Ab}^-})$.  Here, $\hbar$ indicates the action of an additional torus $\Cs_{\hbar}$.  To an NS5 brane $\Zb$ we associate the ``two-way part''
\begin{equation}
    \label{eq: two way part}
    \MM_{\Zb}= \h\Hom(W_{{\Zb}^+},W_{{\Zb}^-}) \oplus
\Hom(W_{{\Zb}^-},W_{{\Zb}^+}),
\end{equation}
whose elements will  be denoted by $(C_{\Zb}, D_{\Zb})$. To a D3 brane $\Xb$ we associate $\NN_{\Xb}=\h\End(W_{\Xb})$. In these formulas, the $\h$ factor means an action of an extra $\Cs$ factor called $\Cs_{\h}$.
Let 
\[
\MM=\bigoplus_{\Ab} \MM_{\Ab} \oplus \bigoplus_{\Zb} \MM_{\Zb}, \qquad \NN=\bigoplus_{\Ab} \NN_{\Ab} \oplus \bigoplus_{\Xb} \NN_{\Xb}.
\]
We define a map $\mu:\MM \to \NN$ componentwise as follows.
\begin{itemize}
\item 
The $\NN_{\Ab}$-component of $\mu$ is 
$B_{\Ab} A_{\Ab} -A_{\Ab} B'_{\Ab}+a_{\Ab} b_{\Ab}$.
\item
The $\NN_{\Xb}$-components of $\mu$ depend on the diagram:
\begin{itemize}
\item[\ttt{\bs -\bs}] If $\Xb$ is in between two D5 branes then it is $B'_{\Xb^-}-B_{\Xb^+}$. 
\item[\ttt{{\fs}-{\fs}}] If $\Xb$ is in between two NS5 branes then it is $C_{\Xb^+}D_{\Xb^+}$ $-D_{\Xb^-}C_{\Xb^-}$.
\item[\ttt{{\fs}-\bs}] If $\Xb^-$ is an NS5 brane and $\Xb^+$ is a D5 brane then it is $-D_{\Xb^-}C_{\Xb^-}$ $-B_{\Xb^+}$.
\item[\ttt{\bs -{\fs}}] If $\Xb^-$ is a  D5 brane and $\Xb^-$ is an NS5 brane then it is $C_{\Xb^+}D_{\Xb^+}+B'_{\Xb^-}$.
\end{itemize}
\end{itemize}
Let $\widetilde{\mathcal M}$ consist of points of $\mu^{-1}(0)\subset \MM$ for which the stability conditions  
\begin{itemize}
\item[(S1)]  if $S\leq W_{{\Ab}^+}$ is a subspace with $B'_{\Ab}(S)\subset S$, $A_{\Ab}(S)=0$, $b_{\Ab}(S)=0$ then $S=0$,
\item[(S2)]  if $T \leq W_{{\Ab}^-}$ is a subspace with $B_{\Ab}(T)\subset T$, $Im(A_{\Ab})+Im(a_{\Ab})\subset T$ then $T=W_{\Ab^-}$
\end{itemize}
hold. As shown by Takayama \cite[prop. 2.9]{Takayama}, the open subset $\widetilde{\mathcal M}\subset \mu^{-1}(0)$ is affine. Let $G=\prod_{\Xb} \GL(W_\Xb)$ and consider the character 
\begin{equation}\label{eq:character}
\chi: G\ \to \C^{\times}, 
\qquad\qquad
(g_\Xb)_{\Xb} \mapsto \prod_{\Xb'} \det(g_{\Xb'}),
\end{equation}
where the product runs for D3 branes $\Xb'$ such that $(\Xb')^-$ is an NS5 brane (in picture \ttt{{\fs}X'}). The Bow variety $X(\DD)$ is defined as the GIT quotient
\[
X(\DD):=\widetilde{\mathcal{M}}//^{\chi} G.
\]

Let $G$ act on $\widetilde{\mathcal M} \times \C$ by $g.(m,x)=(gm,\chi^{-1}(g)x)$. We say that $m\in \widetilde{\mathcal M}$ is semistable (notation $m\in \widetilde{\mathcal M}^{ss}$) if the closure of the orbit $G(m, x)$ in $\widetilde{\mathcal{M}}\times \C$ does not intersect with $\widetilde{\mathcal{M}}\times\{0\}$ for any (every) $x\neq 0$. We say that $m$ is stable (notation $m\in \widetilde{\mathcal M}^{s}$) if $G(m,x)$ is closed and the stabilizer of $(m,x)$ is finite for $x\neq 0$. Although a priori different, semistability and stability coincide for the chosen character $\chi$. This fact can be proved by transposing Nakajima's original argument \cite[Section 3]{Nak98} in the setting of bow varieties. Since the stabilizers are trivial \cite[Lemma 2.10]{Nakajima_Takayama}, the GIT quotient map $\widetilde{\mathcal{M}}^{ss}\to X(\DD)$ is a $G$-torsor. Consequently, $X(\DD)$  coincides with the orbit space $\widetilde{\mathcal{M}}^{ss}/G=\widetilde{\mathcal{M}}^{s}/G$. In particular, $X(\DD)$ is smooth.

Let $n$ be the number of D5 branes in $\DD$. Beyond the complex structure, the bow variety $X(\DD)$ carries a symplectic structure\footnote{We do not make use of the symplectic structure in this article.} as well as the action of a $n+1$ dimensional torus \cite{Nakajima_Takayama}. The torus $\Tt$ acting on $X(\DD)$ can be decomposed as $\Tt=\At\times \C_{\h}^\times$, where $\At=\prod_{i=1}^n \GL(\C_{\Ab_i})$ is a rank $n$ torus rescaling the spaces $\C_{\Ab_i}$ in $\MM_{\Ab_i}$, cf. \eqref{eq: three way part}, and $\C_{\h}^\times$ rescales the arrows of \eqref{eq: three way part} and \eqref{eq: two way part} as prescribed by the weight $\hbar$. Since both actions preserve the zero locus $\mu^{-1}(0)$, they descend to the GIT quotient $X(\DD)$.

\subsection{Torus-fixed points}

Let $\DD$ be a brane diagram. A {\em tie diagram} of $\DD$ is a set of pairs $(\Zb,\Ab)$ where $\Zb$ is an NS5 brane and $\Ab$ is a D5 brane. We indicate such a pair by a curve (`tie') connecting the two fivebranes. It is aesthetic to draw the curve above the diagram if $\Zb$ is to the left of $\Ab$, and below otherwise. It is also required that {\em each D3 brane is covered by these curves as many times as its multiplicity}. Here are two of the 123 different tie diagrams of the brane diagram of Section~\ref{sec:branediagrams}:
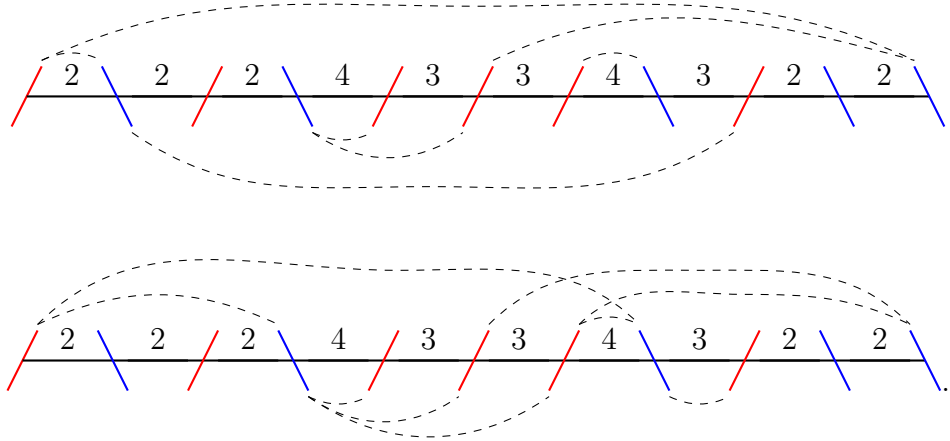
\[
\begin{tikzpicture}[baseline=0,scale=.4]
\begin{scope}[yshift=0cm]
\draw [thick,red] (0.5,0) --(1.5,2); 
\draw[thick] (1,1) --(2.5,1) node [above] {$2$} -- (31,1);
\draw [thick,blue](4.5,0) --(3.5,2);  
\draw [thick](4.5,1)--(5.5,1) node [above] {$2$} -- (6.5,1);
\draw [thick,red](6.5,0) -- (7.5,2);  %%
\draw [thick](7.5,1) --(8.5,1) node [above] {$2$} -- (9.5,1); 
\draw[thick,blue] (10.5,0) -- (9.5,2);  
\draw[thick] (10.5,1) --(11.5,1) node [above] {$4$} -- (12.5,1); 
\draw [thick,red](12.5,0) -- (13.5,2);   %%
\draw [thick](13.5,1) --(14.5,1) node [above] {$3$} -- (15.5,1);
\draw[thick,red] (15.5,0) -- (16.5,2);  %%
\draw [thick](16.5,1) --(17.5,1) node [above] {$3$} -- (18.5,1);  
\draw [thick,red](18.5,0) -- (19.5,2);  %%
\draw [thick](19.5,1) --(20.5,1) node [above] {$4$} -- (21.5,1);
\draw [thick,blue](22.5,0) -- (21.5,2);
\draw [thick](22.5,1) --(23.5,1) node [above] {$3$} -- (24.5,1);  
\draw[thick,red] (24.5,0) -- (25.5,2); 
\draw[thick] (25.5,1) --(26.5,1) node [above] {$2$} -- (27.5,1);
\draw [thick,blue](28.5,0) -- (27.5,2);  %%
\draw [thick](28.5,1) --(29.5,1) node [above] {$2$} -- (30.5,1);   
\draw [thick,blue](31.5,0) -- (30.5,2);   %%
\draw [dashed, black](1.5,2.2) to [out=25,in=155] (3.5,2.2);
\draw [dashed, black](19.5,2.2) to [out=25,in=155] (21.5,2.2);
\draw [dashed, black](1.5,2.2) to [out=25,in=180] (16,4) to [out=0,in=155] (30.5,2.2);
\draw [dashed, black](16.5,2.2) to [out=25,in=165] (30.5,2.2);
\draw [dashed, black](10.5,-.2) to [out=-25,in=-155] (12.5,-.2);
\draw [dashed, black](10.5,-.2) to [out=-35,in=-145] (15.5,-.2);
\draw [dashed, black](4.5,-.2) to [out=-35,in=180] (14.4,-2) to [out=0,in=-145] (24.5,-.2);
\end{scope}
\end{tikzpicture}
\]
\[
\begin{tikzpicture}[baseline=0,scale=.4]
\begin{scope}[yshift=0cm]
\draw [thick,red] (0.5,0) --(1.5,2); 
\draw[thick] (1,1) --(2.5,1) node [above] {$2$} -- (31,1);
\draw [thick,blue](4.5,0) --(3.5,2);  
\draw [thick](4.5,1)--(5.5,1) node [above] {$2$} -- (6.5,1);
\draw [thick,red](6.5,0) -- (7.5,2);  %%
\draw [thick](7.5,1) --(8.5,1) node [above] {$2$} -- (9.5,1); 
\draw[thick,blue] (10.5,0) -- (9.5,2);  
\draw[thick] (10.5,1) --(11.5,1) node [above] {$4$} -- (12.5,1); 
\draw [thick,red](12.5,0) -- (13.5,2);   %%
\draw [thick](13.5,1) --(14.5,1) node [above] {$3$} -- (15.5,1);
\draw[thick,red] (15.5,0) -- (16.5,2);  %%
\draw [thick](16.5,1) --(17.5,1) node [above] {$3$} -- (18.5,1);  
\draw [thick,red](18.5,0) -- (19.5,2);  %%
\draw [thick](19.5,1) --(20.5,1) node [above] {$4$} -- (21.5,1);
\draw [thick,blue](22.5,0) -- (21.5,2);
\draw [thick](22.5,1) --(23.5,1) node [above] {$3$} -- (24.5,1);  
\draw[thick,red] (24.5,0) -- (25.5,2); 
\draw[thick] (25.5,1) --(26.5,1) node [above] {$2$} -- (27.5,1);
\draw [thick,blue](28.5,0) -- (27.5,2);  %%
\draw [thick](28.5,1) --(29.5,1) node [above] {$2$} -- (30.5,1);   
\draw [thick,blue](31.5,0) -- (30.5,2);   %%
\draw [dashed, black](1.5,2.2) to [out=25,in=155] (9.5,2.2);
\draw [dashed, black](1.5,2.2) to [out=45,in=180] (13.5,4) to[out=0,in=135] (21.5,2.2);
\draw [dashed, black](16.5,2.2) to [out=45,in=180] (23.5,4) to [out=0,in=135] (30.5,2.2);
\draw [dashed, black](19.5,2.2) to [out=45,in=180] (25,3.2) to [out=0,in=155] (30.5,2.2);
\draw [dashed, black](19.5,2.2) to [out=25,in=155] (21.5,2.2);
\draw [dashed, black](10.5,-.2) to [out=-25,in=-155] (12.5,-.2);
\draw [dashed, black](10.5,-.2) to [out=-35,in=-145] (15.5,-.2);
\draw [dashed, black](10.5,-.2) to [out=-35,in=-145] (18.5,-.2);
\draw [dashed, black](22.5,-.2) to [out=-25,in=-155] (24.5,-.2);
\end{scope}
\end{tikzpicture}.
\]
Given a torus $A\subset \Tt$, we denote by $X(\DD)^A$ the $A$-fixed locus. 
\begin{theorem}\cite{RimanyiShou, BottaRimanyi}
The natural inclusion $X(\DD)^{\Tt}\hookrightarrow X(\DD)^{\At}$ is the identity map. The set $X(\DD)^{\Tt}$ (equivalently, $X(\DD)^{\At}$) is finite. It is in bijection with the set of tie diagrams of $\DD$. It is also in bijection with $\BCT(r,c)$ where $r$ and $c$ are the charge vectors of $\DD$.
\end{theorem}

Let us sketch the explicit bijections from \cite{RimanyiShou}. A fixed point on $\Ch(\DD)$ is described by a representative in $\MM$, that is, by a collection of matrices $A_{\Ab},B_{\Ab},B'_{\Ab},a_{\Ab},b_{\Ab},C_{\Zb},D_{\Zb}$. These matrices turn out to be ``combinatorial'', having only $0$, $1$, $-1$ entries. The tie diagram is an economic way of describing the positions of the non-zero entries---the details, under the name of ``butterfly diagrams'' are in \cite[Sect.~4.3]{RimanyiShou}.

The bijection between the set of tie diagrams and BCTs is as follows. The $(\Zb, \Ab)$-entry of the BCT corresponding to a tie diagram is 
\[
\begin{cases}
1 & \text{ if } (\Zb-\Ab \text{ is a tie and $\Zb$ is left of $\Ab$}) \text{ or }
(\Zb-\Ab \text{ is not a tie and $\Zb$ is right of $\Ab$})
\\
0 & \text{ if } (\Zb-\Ab \text{ is a tie and $\Zb$ is right of $\Ab$}) \text{ or }
(\Zb-\Ab \text{ is not a tie and $\Zb$ is left of $\Ab$}).
\end{cases}
\]
For example, the BCTs corresponding to the two tie diagrams above are
\[
\begin{pmatrix}
1 & 0 & 0 & 0 & 1 \\
1 & 0 & 0 & 0 & 0 \\
1 & 0 & 0 & 0 & 0 \\
1 & 0 & 0 & 0 & 1 \\
1 & 1 & 1 & 0 & 0 \\
0 & 1 & 1 & 0 & 0 
\end{pmatrix},
\begin{pmatrix}
0 & 1 & 1 & 0 & 0 \\
1 & 0 & 0 & 0 & 0 \\
1 & 0 & 0 & 0 & 0 \\
1 & 0 & 0 & 0 & 1 \\
1 & 0 & 1 & 0 & 1 \\
1 & 1 & 0 & 0 & 0 
\end{pmatrix}
\in
\BCT((2,1,1,2,3,2),(5,2,2,0,2)).
\]

\medskip

\noindent{\em Assumption.} In the whole paper we assume that $\Ch(\DD)$ has at least one fixed point. Equivalently, that for the associated charge vectors $\r,\c$ the set $\BCT(\r,\c)$ is not empty.

 %Without loss of generality, we can we assume that $d_{\Xb}> 0$ for all D3 brane $\Xb$ of $\DD$. Indeed, a brane diagram with exactly $n$ multiplicities $d_{\Xb}=0$ corresponds to a disjoint union of $n+1$ bow varieties.

\subsection{Hanany-Witten transition}
\label{section: hanany witten}
The local surgery on a brane diagram 
\[
\ttt{$d_1$\fs $d_2$\bs $d_3$} \leftrightarrow \ttt{$d_1$\bs $d_1+d_3-d_2+1$\fs $d_3$}
\]
(in either direction) is called Hanany-Witten (HW) transition, cf. \cite{HW}. The charge vectors form a complete invariant of the HW equivalence class of brane diagrams. Given a diagram $\DD$ with D5 charge vector $c$ and NS5 charge vector $r$, the separated diagram $\DD'$ in its HW equivalence class is
\[
\DD'=\ttt{\fs $r_1$\fs $r_1+r_2$\fs $r_1+r_2+r_3$\fs \ldots \bs $c_{n-2}+c_{n-1}+c_n$\bs $c_{n-1}+c_n$\bs $c_n$\bs}.
%\DD=\ttt{\fs r_1\fs r_1+r_2\fs \ldots \bs 3\bs \c_m\bs}
\]
In particular, the unique D3 brane $\Xb$ of $\DD'$ lying in between a D5 brane and an NS5 brane has dimension vector $d_{\Xb}=\sum_i^n c_i=\sum_j^m r_j$. Here, and also in the rest of the article, the integers $n$ and $m$ will denote, respectively, the number of D5 and NS5 branes in a given brane diagram.

The HW transition induces an explicit isomorphism between the associated bow varieties. Let the D5 brane involved in the HW transition be $\Ab$ (ie. the blue $\ttt{\bs}$ brane that moves over an NS5 brane in the surgery). The HW isomorphism is equivariant for the torus $\Tt=\At \times \Cs_\hbar$ along the automorphism of $\Tt$ obtained by re-scaling $\Cs_{\Ab}$ by $\hbar\in \Cs_{\h}$. % (or $\hbar^{-1}$). 

Since the $\Cs_{\h}$ action will not concern us, we will not distinguish between $\Ch(\DD)$ varieties associated to HW equivalent brane diagrams, and use the notation $\X(r,c)$ for $\Ch(\DD)$ for any $\DD$ with charge vectors $r$ and $c$. If we wish, we can choose the separated representative.

\section{The geometric Bruhat order}

\subsection{Geometric Bruhat order}
Let us choose a permutation $\sigma\in S_n$. It induces a cocharacter
\[
\lambda_\sigma: \C^{\times}\to \At \times \C_\hbar^\times,\qquad\qquad t\mapsto (t^{\sigma(1)},t^{\sigma(2)},\ldots,t^{\sigma(n)},1).
\]
For a torus fixed point $f\in \Ch(\DD)$ we define the corresponding attractive manifold
\[
\Attr_{\sigma}(f)=\{ x\in \Ch(\DD) : \lim_{t\to 0} \lambda_\sigma(t)\cdot x = f\}.
\]
\begin{remark}
\label{rem:generic cocharacter}
Since the tangent $\At$-weights at the fixed points in $X(\DD)$ are of the form ${a_i}/{a_j}$, their pullback by $\lambda_\sigma$ are nontrivial. As a consequence, the $\Cs$-action on $X(\DD)$ is generic, i.e. we have $X(\DD)^{\Cs}=X(\DD)^{\At}$.
\end{remark}

For torus fixed points $f,g \in \Ch(\DD)^{\At}$ we define 
\[
f\leq g\qquad\qquad \text{if}\qquad\qquad f\in \overline{\Att{\sigma}(g)},
\]
and its transitive closure will be called the {\em geometric Bruhat order} and denoted by $\leqgb$. 
\begin{remark}
    This definition of geometric partial order applies to any smooth quasiprojective variety $X\subseteq \Pe^N $ equipped with the action of a torus $\At$ and does not exploit extra geometric structure carried by bow varieties. In particular, the fact that this is really a partial order, i.e. that $f\leqgb g $ and $g\leqgb f$ imply $f=g$, follows from smoothness, which in turn ensures that the embedding $X\subseteq \Pe^N$ can be made $\At$-equivariant \cite[Theorem 5.1.25]{Chriss_book}.
\end{remark}

Notice that $\leqb$, $\leqsb$, $\leqgb$ are partial orders on the same set $\BCT(\r,\c)$, the last one depending on the choice of the permutation $\sigma$. 

\begin{remark} \label{GenoaCFC}
In the $c=1^n$ special case, which is relevant in Schubert Calculus, the three partial orders are the same (where $\leqgb$ is meant with the identity permutation). The coincidence of the first two is proved in \cite{BrualdiDeaett, Fernandes2019ClassesO} and the equality of $\leqb$ and $\leqgb$ is the content of \cite[Thm.5A]{proctor}.
\end{remark}

The main theorem of our paper is the identification of the geometric Bruhat order with one of the combinatorial ones, for arbitrary $r,c$. 

\begin{theorem} \label{thm:main}
The geometric Bruhat order for the identity permutation is the same as the secondary combinatorial Bruhat order.
\end{theorem}

The theorem implies that the geometric Bruhat order for another permutation $\sigma$ coincides with an appropriate modification of the combinatorial secondary Bruhat order---one where a $\sigma$-permutation of the columns is taken into account. 
In Sections \ref{sec:proof_invariant_curves} and \ref{sec:proof_resolution} we will give two different proofs for Theorem \ref{thm:main}.

\medskip

\noindent{\em Assumption.} In theory the vectors $r$ and $c$ can have 0 components. However, simply erasing those components effects none of our notions or statements. Hence, from now on in the whole paper we assume that $r$ and $c$ have no 0 components. 

\subsection{More on the geometric Bruhat order}

Since the union of all attracting sets in a bow variety is closed, cf. \cite[Lemma 5.12]{BottaRimanyi}, the geometric Bruhat order admits the following equivalent characterization: $f\leq_G g$ iff there exists a collection of $\At$-fixed points $f=f_0,f_1, \dots, f_{n-1}, f_n=g$ connected by a chain of attracting $\At$-orbits. Namely, if there exist $\At$-orbits $\Or_i\subseteq X(\DD)$ such that $f_i,f_{i-1}\in \overline{\Or_i}$ and 
\[
\lim_{t\to0} \lambda_{\sigma}(t)\cdot x=f_{i-1} \qquad \forall x\in \Or_i.
\]
In general, the dimension of these orbits can be up to $\text{rk}(\At)-1$\footnote{In principle, it could be up to $\text{rk}(\At)$, but the one dimensional torus consisting of elements the form $(a,\dots,a)$ acts trivially on the bow variety.} However, in Lemma \ref{lma:one dimensional orbits} we will show that we can restrict ourselves to closures of \emph{one-dimensional} orbits. Before proving this result, let us discuss its consequences.  By an easy adaptation of Lemma \ref{lemma curves in varieties}, it follows that one dimensional orbit closures admit the following evocative description: two distinct points $f,g\in X^{\At}$ belong to the closure $\overline \Or\subset X(\DD)$ of a one dimensional $\At$-equivariant orbit $\Or$ iff there exists an embedding $\gamma: \Pe^1\to \overline \Or$ whose restriction to $\Cs=\Pe^1\setminus \lbrace 0, \infty\rbrace$ is onto $\Or$ and such that $\gamma(0)=f$, $\gamma(\infty)=g$. In particular, this implies
\[
    \lim_{t\to 0} \lambda_\sigma(t)\cdot x= f \qquad \lim_{t\to \infty} \lambda_\sigma(t)\cdot x=  g
\]
for every $x\in \Or$. Hence, a compact one dimensional $\At$-orbit closure can be thought of as a $\Pe^1\subseteq X(\DD)$ preserved by the $\At$-action and whose poles $0,\infty\in \Pe^1$ are $\At$-fixed. We call such curves $\At$-invariant curves. For convenience, we always parametrize the domain in such a way that $\gamma(\infty)\leq_{G} \gamma(0)$. This in particular implies that the tangent weight at $0\in \Pe^1$ of an $\At$-invariant curve is always attractive, i.e. its pullback by $\lambda_\sigma$ is of the form $\lambda_{\sigma}^* \chi(t)=t^k$ for some $k>0$. 

\begin{remark}
    Not all one dimensional orbit closures are compact. For example, consider the bow variety $X(\ttt{\fs 1\fs 2\bs 1\bs })=\TsP^1$. The fibers $T_0^*\Pe^1$ and $T_\infty^*\Pe^1$ are one dimensional orbit closures, with boundaries given by the base-points $0,\infty\in \Pe^1$.
    A similar situation can be observed in Figure~\ref{fig:212}, where both compact and non-compact invariant curves are drawn. In particular, notice that non-compact invariant curves do not affect the geometric Bruhat order. This is indeed a general fact, which directly follows from the characterization of the geometric Bruhat order discussed above.
\end{remark}

We now characterize the geometric Bruhat order in term of one dimensional attracting curves.

\begin{lemma}
\label{lma:one dimensional orbits}
    Let $f,g\in X(\DD)^{\At}$. Then $f\leq_G g$ iff there exist $\At$-fixed points $f=f_0,f_1, \dots, f_n=g$ and a chain of $\At$-invariant curves $\gamma_i:\Pe^1\to X(\DD)$ such that $\gamma_i(0)=f_{i-1}$ and $\gamma_i(\infty)=f_i$ for all $i=1,\dots, n$
\end{lemma}

\begin{proof}
    To save on notation, set $X=X(\DD)$. Notice that the geometric Bruhat order only depends on $\At$ via the character $\lambda: \Cs\to \At$. Since the latter is generic in the sense of Remark \ref{rem:generic cocharacter}, it follows that $f\leq_G g$ iff $f$ is connected to $g$ via a chain of a $\Cs$-invariant curves. As a consequence, it suffices to show that an attractive $\Cs$-invariant curve $\gamma: \Pe^1\to X$ degenerates to a chain of attractive $\At$-invariant curves. As we will see in the proof, the existence of such degeneration requires some amount of compactness. Although bow varieties are non-compact in general, they are defined as GIT quotients of some affine variety with respect to the action of a reductive group. As a consequence, any bow variety $X$ comes with a canonical proper $\Tt$-equivariant map $\pi: X\to X_0=\text{Spec}(\Gamma(\mathcal{O}_{X}))$. This relative compactness will be sufficient for the proof. 
    
    Consider the composition $ \mathbb{P}^1\xrightarrow{\gamma} X\xrightarrow{\pi} X_0$. Since $X_0$ is affine, $\pi\circ \gamma$ is constant. Let $x\in X_0^{\At}$ be its image. Since $\pi$ is equivariant, the fiber $Z:=\pi^{-1}(x)$ is preserved by $\At$ and contains the image of $\gamma$. Moreover, since $\pi$ is proper, $Z$ is complete. As a result, to prove the lemma it suffices to show that the $\Cs$-invariant curve $\gamma: \Pe^1\to Z\subset X$ degenerates to a chain of attractive $\At$-invariant curves inside $Z$, which is complete.
    
    Let $\chi_1, \dots \chi_d$ be set of attractive tangent weights of $X^{\At}\subset X$. This set is really finite because the $X^{\At}$ is a finite set and is nonempty unless $X$ is zero dimensional, in which case the lemma is trivial.
    Set $\At_{d+1}:= \At$, $\At_k:=\cap_{j=k}^{d} \text{Ker}(\chi_j)\subseteq \At$ for all $k=1,\dots, d$ and choose splittings $\At_k=\At_{k-1}\times \At_k/\At_{k-1}$. Notice that $\At_1\subset\dots \subset \At_d\subset\At$, the tori $ \At_k/\At_{k-1}$ are one dimensional, and $\At_1$ acts trivially on $X$. The cocharacter $\lambda$ and the weight $\chi$ factor as
    \[
    \begin{tikzcd}
    \Cs\arrow[r, "\lambda"]\arrow[d, , swap, "\Delta"] & \At_{d+1}\\
    \Cs\times \Cs\arrow[r, swap, "\lambda'\times \lambda''"]&  \At_{d}\times \At_{d+1}/\At_d \arrow[u, equal]
    \end{tikzcd}
    \qquad 
    \begin{tikzcd}
        \At_{d+1}\arrow[r, "\chi_d"]\arrow[d] & \Cs\\
        \At_{d+1}/\At_d\arrow[ur]
    \end{tikzcd}
    \]
    Consider the induced action of $\Cs\times \Cs$ on $Z\subset X$ and fix a point $ t \in \Pe^1\setminus \lbrace 0,\infty\rbrace  $. Since $Z$ is complete, Lemma \ref{lemma fixed  points rank two torus} (i) provides a map 
    \[
    \Gamma: \Pe^1\times \Pe^1\xrightarrow[]{p^{-1}} \overline{\Graph(\gamma(t))}\to X
    \]
    such that $\gamma$ is recovered as the composition $\Pe^1 \xrightarrow{\Delta} \Pe^1\times \Pe^1\to X$.
    Moreover, since $\lambda$ is generic, the same holds for $\lambda'$ and $\lambda''$. As a consequence, Lemma \ref{lemma fixed  points rank two torus} (ii)-(iii) implies that $\gamma_d: \lbrace 0 \rbrace\times \Pe^1\to Z\subset X$ is fixed by $\At_d$ and its image is the closure of an $\At_{d+1}/\At_d$ orbit, which must be either constant or one dimensional with tangent weight $\chi_d$. Hence, combining the two diagrams above we conclude that $\gamma_d$ is either constant or an attractive $\At$-invariant curve (with tangent weight $\chi_d$ at $0\in \Pe^1$). Similarly, we see that the map $\gamma': \Pe^1\times \lbrace \infty\rbrace\to Z\subset X$ is either constant or an attractive $\Cs$-invariant curve for the $\Cs$-action induced by the character $\lambda': \Cs\to \At_d$.
    
    Overall, we have shown that we may replace the curve $\gamma$ with the concatenation of $\gamma_d$ and $\gamma'$. Iterating the argument for all pairs $(\At_k,\At_{k+1})$, we see that we may replace $\gamma$ by a chain of attracting orbits $\gamma_k: \Pe^1\to Z\subset X$ (with tangent weight $\chi_k$ at $0\in \Pe^1$) connecting $f$ to $g$\footnote{Notice that in general $n\leq d$ because some maps $\gamma_i: \Pe^1\to X$ might be constant, so we can neglect them.}. This proves the lemma. 

\end{proof}

%{\color{red}(Tommaso) [How trivial is the fact that the geometric order defined via closures of attracting sets is equivalent to the one realized by considering chains of $\Tt$-invariant curves (BTW, where is the definition of a $\Tt$-invariant curve? Also, $\Tt$-equivariant or $\At$-equivariant)? I cannot envision a proof of the equivalence that does not exploit some structure of bow varieties. For instance, I believe that properness and equivariance of the affinization map $\pi: X\to X_0$] suffice. Actually, I think that it's easy to produce a counterexample by dropping this assumption. Perhaps we should explain this.}

\begin{remark}
     Since every $\At$-fixed point in a bow variety is also $\Tt$-fixed, cf. \cite[Section 4]{RimanyiShou}, the argument above actually shows that the $\Tt$-invariant curves are sufficient to characterize the geometric Bruhat order as in Lemma \ref{lma:one dimensional orbits}.
\end{remark}

 For completeness, we give a counterexample to stress that the argument above fails in absence of completeness.

\begin{counterexample}
 Consider the bow variety $X=\TsGr(2,4)$, and set $Y=\text{Gr}(2,4)$. The affinization map  $\pi: X\to X_0$  is the Springer resolution of a nilpotent orbit in $X_0\subset \mathfrak{gl}_4$ and all the fixed points $X^{\At}$ are contained in the zero fiber $\pi^{-1}(0)=Y$.

The $\At$-fixed points $Y^{\At}=X^{\At}$ are the coordinate planes $f_{ij}=\C_{a_i}\oplus \C_{a_j}\subset \C^4$ for $1\leq i<j\leq4$. Set $Y^0=Y\setminus \lbrace f_{13},f_{14}, f_{23}, f_{24} \rbrace$ and let $\lambda_{id}: \C^{\ast}\to \At$ be the character $t\mapsto (t,t^2,t^3,t^4)$. Unlike $Y$, $Y_0$ is not complete and hence the argument above fails. Indeed, it is easy to check that there exists a $\C^{\ast}$-invariant curve connecting $f_{12}$ to $f_{34}$ in $Y^0$ (and hence $f_{24}\in \overline{\text{Att}_{\sigma}(f_{12})}$ or vice versa) but no chain of $\At$-invariant curves. On the contrary, this chain of curves exists in $X$, and consists of two copies of $\mathbb{P}^1$ glued at a pole. In particular, the pole at the intersection of the two $\mathbb{P}^1$ in $X$ must be one of the fixed points $ \lbrace f_{13},f_{14}, f_{23}, f_{24}\rbrace $, which are not contained in $Y^0$.
\end{counterexample}

\smallskip

\section{The proof via combinatorial description of invariant curves}
\label{sec:proof_invariant_curves}

\subsection{Classification of Compact Invariant Curves}
\par For a bow variety $\Ch(\DD)$, the $\Tt$-invariant curves can be read directly from the BCTs of fixed points. First, we restate the classification from \cite{FosterShou} for the compact invariant curves. 

\begin{definition}
    Let $M\in\BCT(r,c)$. A submatrix $\overline{M}=\begin{pmatrix}
        m_{pq}
    \end{pmatrix}^{p=i,i+1,\dots,j}_{q=k,l}$ with only two columns is called a matched block if $\sum_{p=i}^{j}m_{pk}=\sum_{p=i}^{j}m_{pl}$. We call $\overline{M}$ a minimal matched block if it cannot be partitioned into two matched sub-blocks.
\end{definition}

\begin{definition}
    Let $M\in\BCT(r,c)$. A block swap move $\psi$ on $M$ consists of swapping the columns of a matched block $\overline{M}$ within $M$. If $\overline{M}$ is minimal, then $\psi$ is called indecomposable.
\end{definition}

Note that if a minimal matched block has top row equal to $(0,1)$ then the bottom row is $(1,0)$, and vice-versa. Moreover, every block swap move uniquely decomposes into simultaneous indecomposable block swap moves with minimal matched blocks separated by $(0,0)$ and $(1,1)$ rows. For example, the following matched block decomposes into three minimal matched blocks:
\[\begin{pmatrix}
0 & 1\\
1 & 0\\
\hline
1 & 1\\
\hline
1 & 0\\
1 & 0\\
0 & 1\\
1 & 1\\
0 & 1\\
\hline
0 & 0\\
\hline
1 & 0\\
0 & 1\\
\end{pmatrix}\text{.}\]

Before classifying invariant curves in $\Ch(\DD)$, it is important to note that bow varieties generally contain pencils of invariant curves, by which we mean a family of invariant curves whose tangent directions at a fixed point span one weight space. For an example, consider $(\mathbb{P}^1)^k$ with the diagonal action by $\Cs$. In fact, all compact invariant pencils in a bow variety are of this form~\cite{FosterShou}. Also pencils can occur non-compactly, as in Figure \ref{fig:212}, to no effect on $\leqgb$.

\begin{theorem} \cite{FosterShou} \label{thm:classification}
Let $\DD$ be a separated brane diagram with charge vectors $r$ and $c$.
\begin{enumerate}
    \item The block swap moves in $\BCT(r,c)$ are in bijection with the pencils of compact invariant curves in $\Ch(\DD)$.
    \item If a block swap move decomposes into $k$ simultaneous indecomposable block swap moves, then the associated pencil has dimension $k$. 
    \item For a block swap move from $M$ to $M^\prime$, let $i$ be the index of the topmost affected row. Let $(i,q_0)$ be the index of the $0$ that becomes a $1$ and let $(i,q_1)$ be the index of the $1$ that becomes a $0$. Then the associated pencil has tangent weight $\frac{a_{q_1}}{a_{q_0}}\hbar^{1+s_{iq_0}-s_{iq_1}}$ in $T_M\Ch(\DD)$, where $s_{pq} = \sum_{p^\prime=1}^p m_{p^\prime q}$.
\end{enumerate}
\end{theorem}

\begin{remark}
    These block swap moves contain the same data as ``butterfly surgeries'', which are the natural way to explicitly define these curves and were first described in \cite{Shou}. The verification that these are the only compact invariant curves is completed in \cite{FosterShou}.
\end{remark}

% \par There is also an intermediate interpretation of block swap moves as Young diagram surgeries, made by cutting and pasting certain pieces of Young diagrams generated using the columns of a BCT. The study of BCTs, matched blocks, and some spin-offs of these surgeries is also relevant to the classification of unbounded invariant curves. See \cite{FosterShou} for more information.

%{\color{red}(Tom)[If I understand correctly, the $\lbrace\text{combinatorial}\implies\text{geometric}\rbrace$ part of the proof below does not require the whole power of Theorem \ref{thm:classification} (1) because it suffices to exhibit \emph{some} curve connecting $M$ to $M'$ whenever $M^\prime\lessdot_{\hat{B}}M$. I believe it it's fairly easy to explicitly produce a curve that does the job: is enough to cook up a ``one parameter interpolation'' of the quiver representations defining the two fixed points $M$ and $M'$. Is it worth mentioning this?]}

%{\color{teal}(Alex)[I agree these curves, the connected butterfly surgeries, are not too hard to cook up. They are the curves I call $\prec_G$. The only thing that makes Theorem \ref{thm:classification} (1) stronger is that these plus pencils are the \emph{only} bounded fixed curves, which I agree is not needed for the $\lbrace\text{combinatorial}\implies\text{geometric}\rbrace$ half. I added a little more detail above.]}
\subsection{First proof of Theorem \ref{thm:main}}
\label{subsection:first proof main thm}
With compact invariant curves now classified, we return to the geometric Bruhat order on fixed points with greater clarity. First note that we do not need pencils of dimension greater than $1$, as the indecomposable block swap moves are sufficient to generate the geometric order. 
Now define the relation $\prec_G$ such that $M^\prime\prec_G M$ if there exists an indecomposable block swap move from $M$ to $M^\prime$ with tangent weight $\frac{a_{N}}{a_{n}}\hbar^d$ with any $d\in\mathbb{Z}$ and $N>n$. In other words, $\prec_G$ corresponds to the relation on fixed points given by directed invariant curves. Note that $\prec_G$ does generate $\leq_G$ by transitive closure, but it is not a covering relation, as triangles may still appear in $\prec_G$, as in Figure \ref{fig:212}.

\begin{proof}
    Let $\lessdot_{\hat{B}}$ be the covering relation for $\leqsb$. We first show that if $M^\prime\lessdot_{\hat{B}}M$, then $M^\prime \prec_G M$. By Theorem \ref{thm:cover}, the $L_2\to I_2$ interchange on $M[\{i,j\},\{k,l\}]$ yields $M^\prime$, and with condition (1), $m_{pk}=m_{pl}$ for $i<p<j$. This means that $\overline{M}=\begin{pmatrix}
        m_{pq}
    \end{pmatrix}_{q=k,l}^{p=i,i+1,\dots,j}$ is a minimal matched block, as all the middle rows of $\overline{M}$ are either $(0,0)$ or $(1,1)$. By Theorem \ref{thm:classification}, since the $L_2\to I_2$ interchange has the same effect as the block swap move swapping the columns of $\overline{M}$, there is an invariant curve between $M^\prime$ and $M$. It has tangent weight $\frac{a_{l}}{a_{k}}\hbar^{1+s_{ik}-s_{il}}$. Since $l>k$, we have $M^\prime \prec_G M$.

    Next we show if $M^\prime \prec_G M$, then $M^\prime \leqsb M$. Let there be an indecomposable block swap move from $M$ to $M^\prime$ with tangent weight $\frac{a_{l}}{a_{k}}\hbar^d$ with $l>k$. Let this block swap move have minimal matched block $\overline{M}=\begin{pmatrix}
        m_{pq}
    \end{pmatrix}_{q=k,l}^{p=i,i+1,\dots,j}$. By the form of its tangent weight, Theorem \ref{thm:classification} says that the $i$th row is $(0,1)$. By minimality, the $j$th row is $(1,0)$. We seek a sequence of $L_2\to I_2$ interchanges to swap the columns of $\overline{M}$.
   
    We will induct on the number of $(0,1)$ rows in $\overline{M}$. If there is only $1$ row of the form $(0,1)$, then it is the top row, and only one $L_2\to I_2$ interchange is needed. If there are $n+1$ rows of the form $(0,1)$, then find the first row under the $i$th row that is either $(0,1)$ or $(1,0)$, and call its index $t$. Note that if row $t$ is $(1,0)$, then minimality is violated, as $\begin{pmatrix}
        m_{pq}
    \end{pmatrix}_{q=k,l}^{p=i,i+1,\dots,t}$ would be a minimal matched subblock. Thus, row $t$ is $(0,1)$ and all rows between $i$ and $t$ are $(0,0)$ or $(1,1)$. Similarly, there is a lowest row $t^\prime$ of the form $(1,0)$ that is above row $j$, and between $t^\prime$ and $j$ are only $(0,0)$ or $(1,1)$ rows. First, perform the $J_2\to I_2$ interchange on rows $i$ and $j$. The part between $t$ and $t^\prime$ is unchanged and has $n$ rows of $(0,1)$. 
    By inductive hypothesis, a sequence of $L_2\to I_2$ interchanges will do the swap $\begin{pmatrix}
        m_{pq}
    \end{pmatrix}_{q=k,l}^{p=t,t+1,\dots,t^\prime}\to \begin{pmatrix}
        m_{pq}
    \end{pmatrix}_{q=l,k}^{p=t,t+1,\dots ,t^\prime}$.
    Note that this matched block from rows $t$ to $t^\prime$ may not be minimal, in which case break it into minimal matched subblocks and apply the inductive hypothesis multiple times. We now have the total block swap move produced from $L_2\to I_2$ interchanges, so $M^\prime \leqsb M$.
    
    We have proven if $M^\prime\lessdot_{\hat{B}}M$, then $M^\prime \prec_G M$, so by transitive closure, if $M^\prime\leqsb M$, then $M^\prime \leqgb M$. We have proven if $M^\prime \prec_G M$, then $M^\prime \leqsb M$, so by transitive closure, if $M^\prime \leqgb M$, then $M^\prime \leqsb M$.
\end{proof}
Note that Theorem \ref{thm:main} is {\em not} a perfect matching of $L_2\to I_2$ interchanges and invariant curves. Although they do generate the same partial order, they are fundamentally different. We illustrate this phenomenon in Figure \ref{fig:DualGr(2,4)} for the example $\Ch(\ttt{\fs 1\fs 2\fs 3\fs 4\bs 2\bs})$, whose fixed point set is $\BCT((1,1,1,1),(2,2))$. This bow variety is also a Nakajima quiver variety
$\mathcal{N}\Big(\begin{tikzpicture}[baseline=-18pt,scale=.4]
\draw[thick] (0,-1) -- (3,-1); 
\draw[fill] (0,-1) circle (3pt);  
\draw[fill] (1.5,-1) circle (3pt);  
\draw[fill] (3,-1) circle (3pt);  
%\draw[fill] (4.5,-1) circle (3pt);  
\draw[thick] (1.4,-2.1) rectangle (1.6,-1.9);
\draw[thick] (1.5,-1) -- (1.5,-1.9);
%\draw[thick] (3,-1) -- (4.5,-1);
\node at (0,-.5) {$\scriptscriptstyle 1$};
\node at (1.5,-.5) {$\scriptscriptstyle 2$};
\node at (3,-.5) {$\scriptscriptstyle 1$};
%\node at (4.5,-.5) {$\scriptscriptstyle 1$};
\node at (2,-2) {$\scriptscriptstyle 2$};
\end{tikzpicture}\Big)$, 
cf.~\cite[Sect.~6]{RimanyiShou}
%This is the smallest example of a bow variety that contains both invariant curves that are not $L_2\to I_2$ interchanges and $L_2\to I_2$ that are not invariant curves. See Figure \ref{fig:DualGr(2,4)}.
and is the ``3d mirror dual'' of $\TsGr(2,4)$, cf. \cite{BottaRimanyi}.

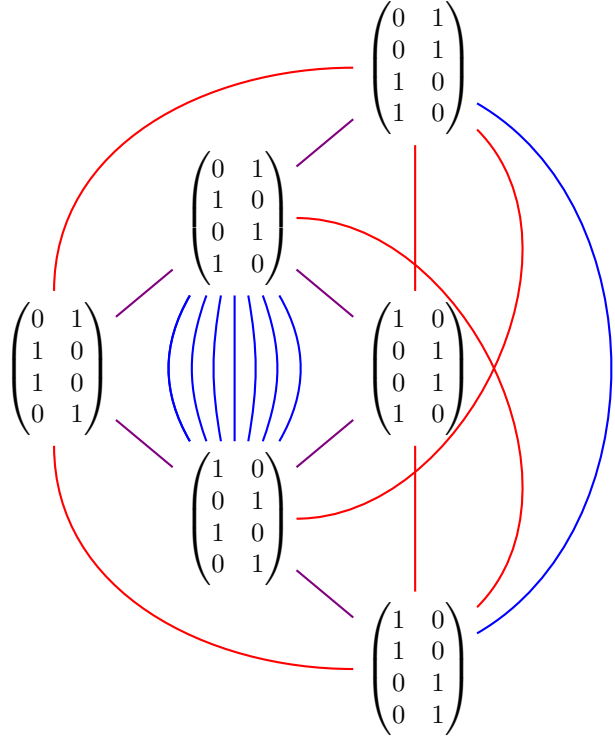
\begin{figure}
\centering
\begin{tikzpicture}[scale=0.8]
\node[] (top) at (3,5) { $\footnotesize\begin{pmatrix} 0&1\\ 0&1\\ 1&0\\ 1&0 \end{pmatrix}$};
\node[] (midup) at (0,2.5) { $\footnotesize \begin{pmatrix} 0&1\\ 1&0\\ 0&1\\ 1&0 \end{pmatrix}$};
\node[] (left) at (-3,0) { $\footnotesize \begin{pmatrix} 0&1\\ 1&0\\ 1&0\\ 0&1 \end{pmatrix}$};
\node[] (right) at (3,0) { $\footnotesize \begin{pmatrix} 1&0\\ 0&1\\ 0&1\\ 1&0 \end{pmatrix}$};
\node[] (middown) at (0,-2.5) { $\footnotesize \begin{pmatrix} 1&0\\ 0&1\\ 1&0\\ 0&1 \end{pmatrix}$};
\node[] (bottom) at (3,-5) { $\footnotesize \begin{pmatrix} 1&0\\ 1&0\\ 0&1\\ 0&1 \end{pmatrix}$};

\draw [thick,violet] (top)--(midup);
\draw [thick,violet] (left)--(midup);
\draw [thick,violet] (right)--(midup);
\draw [thick,violet] (left)--(middown);
\draw [thick,violet] (right)--(middown);
\draw [thick,violet] (bottom)--(middown);

\draw [thick,blue] (bottom) to [out=30,in=-30] (top);
\draw [thick,blue] (midup) to (middown);
% \draw [thick,blue] (midup) to [out=-50,in=50] (middown);
% \draw [thick,blue] (midup) to [out=-130,in=130] (middown);
\draw [thick,blue] (midup) to [out=-60,in=60] (middown);
\draw [thick,blue] (midup) to [out=-120,in=120] (middown);
\draw [thick,blue] (midup) to [out=-70,in=70] (middown);
\draw [thick,blue] (midup) to [out=-110,in=110] (middown);
\draw [thick,blue] (midup) to [out=-120,in=120] (middown);
\draw [thick,blue] (midup) to [out=-80,in=80] (middown);
\draw [thick,blue] (midup) to [out=-100,in=100] (middown);

\draw [thick,red] (top) to (right);
\draw [thick,red] (bottom) to (right);
\draw [thick,red] (top) to [out=180,in=90] (left);
\draw [thick,red] (bottom) to [out=180,in=-90] (left);
\draw [thick,red] (top) to [out=-45,in=0] (middown);
\draw [thick,red] (bottom) to [out=45,in=0] (midup);
\end{tikzpicture}
\caption{A graph with vertex set $\BCT((1,1,1,1),(2,2))$. A blue edge means a compact invariant curve connecting the corresponding fixed points on the bow variety, and the comb of blue edges indicate a 1-parameter family of such curves. A red edge indicates an  $L_2\to I_2$ interchange. A purple edge means we have both an invariant curve and an interchange.}
%In purple are invariant curves that are also $L_2\to I_2$ interchanges. In blue are invariant curves that are not interchanges. In red are interchanges that are not invariant curves.}
\label{fig:DualGr(2,4)}
\end{figure}

\section{The proof via brane-resolution} \label{sec:proof_resolution}

\subsection{Column resolutions for BCT tables}
Let $\r\in \N^m$ and $\c\in \N^n$ satisfying $\sum_{i=1}^m r_i=\sum_{j=1}^n c_j$. For any $k=1,\dots n$, let $\tilde c\in \N^{n+1}$ be the vector obtained from $c$ by replacing an entry $c_k \geq 2$ with a pair of consecutive entries $c_k',c_k''\in \N_{>0}$ such that $c_k'+c_k''=c_k$. We call this operation a resolution of the charge vector $c$. Let $\BCT(r,\tilde c,c)\subset\BCT(r,\tilde c)$ the subset consisting of those $\widetilde M=(\tilde m_{ij})\in \BCT(r,\tilde c)$ such that $\widetilde m_{ik}+\widetilde m_{ik+1}\in \lbrace 0,1\rbrace$ for all $i=1,\dots, m$. There is a well defined map
\[
q: \BCT(r,\tilde c,c)\to \BCT(r,c),\qquad \widetilde M \to q(\widetilde M)
\]
sending a matrix $\widetilde M=(\widetilde m_{ij})$ to 
\[
q(\widetilde M)_{ij}=\begin{cases}
    \widetilde m_{ij} & \text{if $j<k$}\\
    \widetilde m_{ik}+ \widetilde m_{ik+1}& \text{if $j= k$}\\
    \widetilde m_{ij+1} & \text{if $j> k$}
\end{cases}
\]
Thus, the map $q$ merges the $k$-th and $(k+1)$-th columns. Because of the defining condition of $\BCT(r,\tilde c,c)$, this operation really produces a BCT table for $\BCT(r,c)$. 

The fiber $q^{-1}(M)$ is in one to one correspondence with the $S_{c_k}/(S_{c'_k}\times S_{c_k''})$ ways of producing a BCT table $\widetilde M\in \BCT(r,c)$ by replacing the $k$-th column $M_{\bullet k}$ of $M$ with two adjacent columns that sum to $M_{\bullet k}$. This observation motivates the following definition:

\begin{definition}
    Let $M\in \BCT(r,c)$. The elements $\widetilde M\in q^{-1}(M)\subset \BCT(r,\tilde c,c)$ are called resolutions of the $k$-th column of $M$. 
\end{definition}

\begin{example} Let $M\in \BCT(r,c)$ be of the form
    \[
    M=\left(
    \begin{array}{C|c|C}
      & 1 & \\
      \mbox{$\varheart$} &0 &\mbox{$\vardiamond$} \\
      & 1 & \\
      & 1 & \\
    \end{array}
    \right).
    \]
    The central column is the $k$-th column and the placeholders $\varheart$ and $\vardiamond$ are put in place of the other columns. Notice that $c_k=3$. Let $\tilde c$ be the resolution of $c$ such that $ c'_k=2$ and $ c''_k=1$. Then the set $q^{-1}(M)\subset \BCT(r, \tilde c)$ consists of the three column resolutions 
    \[
    \widetilde M_1=\left(
    \begin{array}{C|cc|C}
     &1&0& \\
      \mbox{$\varheart$} &0&0 &\mbox{$\vardiamond$} \\
     &1&0& \\
     &0&1& \\
    \end{array}
    \right)
    \qquad 
    \widetilde M_2=\left(
    \begin{array}{C|cc|C}
     &1&0& \\
     \mbox{$\varheart$} &0&0 &\mbox{$\vardiamond$} \\
     &0&1& \\
     &1&0& \\
    \end{array}
    \right)
    \qquad 
    \widetilde M_3=\left(
    \begin{array}{C|cc|C}
     &0&1& \\
     \mbox{$\varheart$} &0&0 &\mbox{$\vardiamond$} \\
     &1&0& \\
     &1&0& \\
    \end{array}
    \right).
    \]
    
\end{example}

The following lemma, which directly follows from the definitions, shows that the BCT resolutions introduced above are compatible with the secondary Bruhat order. 

\begin{lemma}
\label{lma: Brhuaht order and resolutions}
    Let $M_1,M_2\in \BCT(r,c)$. Then $M_1 \geqsb M_2$ if and only if for every resolution $\widetilde M_1\in q^{-1}(M_1)\subset \BCT(r,\tilde c,c)$ there exists $\widetilde M_2\in q^{-1}(M_2)$ such that $\widetilde M_1\geqsb \widetilde M_2$.
\end{lemma}

%Our second proof of Theorem \ref{thm:main} is based on the geometrization of the concepts introduced in this section. 
%and the reader is advised to have some familiarity with the ``geometric fusion'' procedure of \cite{BottaRimanyi}. 
%In particular, in the next two subsections we will discuss the geometric counterpart of a BCT resolution, and in Section \ref{section: second proof main thm} we will use the geometric analogue of Lemma \ref{lma: Brhuaht order and resolutions} to reduce the proof of our main theorem to the case of a partial flag variety, in which case the result is already known.

\subsection{D5 resolutions}
Throughout this section, we assume that all bow varieties are separated, i.e. their brane diagrams are of the form $\DD=\ttt{\fs $d_1$\fs $d_2$\fs $\dots$\fs $d_m=h_n$\bs $\dots$\bs $h_{2}$\bs $h_1$\bs }$. The combinatorics of the dimension vectors implies that 
\[
d_k=\sum_{i=1}^k r_i\qquad h_l=\sum_{i=n-l+1}^n c_i
\]
for all $k=1\dots, m$ and $l=1,\dots, n$. 
In particular, constraint $d_m=h_n$ is exactly the one entering in the definition of the set $\BCT(r,c)$.

%it follows that a separated brane diagram $\DD$ is uniquely determined by the charge vectors $r,c$. As a consequence, we will denote a separated bow variety by $\X(r,c)$.

%In this section, we describe the geometric counterpart of the column resolutions introduced in the previous section. 
Let $\DD(r,c)$ be a separated brane diagram with charges $r\in \N^m$ and $c\in \N^n$ and let $\DD(r,\tilde c)$ be the diagram obtained by replacing the charge vector $c$ with a resolution $\tilde c$ of its $k$-th entry. In other words, the diagram $\DD(r,\tilde c)$ is obtained via the local surgery
\begin{center}
\begin{tikzpicture}[scale=.7]
\draw[thick] (-1,0) -- (1,0);
\draw[thick, blue] (0.2,-0.5) -- (-0.2,0.5);
%\node[label={\small $a$}] at (0.2,-1.2) {};
\node[label={\small $c'_k+c''_k$}] at (-0.2,0.5) {};

\draw[stealth-stealth, thick] (2,0) -- (3,0);

\draw[thick] (4,0) -- (7,0);
\draw[thick, blue] (5.2,-0.5) -- (4.8,0.5);
\draw[thick, blue] (6.2,-0.5) -- (5.8,0.5);
%\node[label={\small $a$}] at (5.2,-1.2) {};
%\mode[label={\small $a$}] at (6.2,-1.2) {};
\node[label={\small $c'_k$}] at (4.8,0.5) {};
\node[label={\small $c''_k$}] at (5.8,0.5) {};

\end{tikzpicture}
\end{center}
We call this local modification of the brane diagram a D5 brane resolution.
%\footnote{In \cite{BottaRimanyi}, NS5 brane resolutions are also considered. They correspond to row resolutions of the BCT table}. 
%Notice that, since $\DD(r,c)$ is separated, $\DD(r,\tilde c)$ is also separated. 

Let $\X(r,c)$ and $ \Ch(r,\tilde c)$ be the bow varieties associated with $\DD(r,c)$ and $\DD(r,\tilde c)$ respectively. Notice that, according to our conventions, $\X(r,c)$ is acted on by $\At=(\Cs)^{n}$ while $X(r,\tilde c )$ is acted on by $\widetilde \At=(\Cs)^{n+1}$. Throughout this section, we regard $\At$ as a subtorus of $\widetilde \At$ via the embedding 
\[
\varphi: \At\hookrightarrow \widetilde \At \qquad (a_1,\dots, a_k, \dots, a_n)\mapsto (a_1,\dots, a_k, a_k, \dots, a_n)
\]
and consider its induced action on $X(r,\tilde c)$.
\begin{theorem}{\cite[Section 6.1]{BottaRimanyi}}
\label{theorem:geometric resolution}
There exists an $\At$-equivariant embedding $j:\X(r,c)\hookrightarrow  X(r,\tilde c)$. Moreover, for any fixed point $f\in \X(r,c)^{\At}$ the $\At$-fixed component $F\subseteq X(r,\tilde c )^{\At}$ such that $j(f)\in F$ is isomorphic to the bow variety with brane diagram
\[
\ttt{\fs $1$\fs $2$\fs $3$\fs $\dots$\fs $c'_k+c''_k$\bs $c''_k$\bs}
\]
\end{theorem}
This theorem encodes all the data entering in the combinatorics described in the previous section. Indeed, we have:
\begin{enumerate}
    \item[(i)] the $\At$-fixed points in $\X(r,c)$ are in one to one correspondence with the set $\BCT(r,c)$,
    \item[(ii)] the $\widetilde \At$-fixed points in $X(r,\tilde c)$ are in one to one correspondence with the set $\BCT(r,\tilde c)$,
    \item[(iii)] for any $M\in \BCT(r,c)$ corresponding to some $f\in \X(r,c)^{\At}$, the set $ F^{\widetilde \At}$ is in one to one correspondence with the fiber $q^{-1}(M)\subset \BCT(r,\tilde c)$.
\end{enumerate}

\subsection{Maximal resolutions}

%By iteration of Theorem \ref{theorem:geometric resolution}, one can embed $\X(r,c)$ in a bow variety all whose D5 branes have charge one. Such a bow variety is isomorphic to the cotangent bundle of a partial flag variety.
For any $d\in \N$ such that $d=\sum_{k=1}^m  r_k$, let $\Fl(r,d)$ be the variety parametrizing quotients\footnote{The choice of considering quotients rather than inclusions is due to our choice of stability condition on bow varieties, see \cite{RimanyiShou, BottaRimanyi}.}
\[
{\C}^{d}=V_m\twoheadrightarrow V_{m-1} \twoheadrightarrow\dots \twoheadrightarrow V_{1}\twoheadrightarrow V_{0}=0
\]
where $\dim(V_k)=d_k=\sum_{i=1}^k r_k$. 
In particular, for  $1^d=(1,1,\dots, 1)\in \N^d$ we get the full flag variety $\Fl(1^d, d)$. 
The maximal torus $\widetilde \At\subset \GL(d)$ naturally acts on $\Fl(r,d)$ and hence also on its cotangent bundle $\TsFl(r,d)$. For a given $c\in \N^n$ such that $\sum_{i=1}^m c_i=\sum_{i=1}^m r_i$, let $\At=(\Cs)^n\subset\widetilde \At$ be the subtorus acting on $V_m$ with weight decomposition $V_m={\C}^{d_m}= a_1{\C}^{c_1}\oplus \dots \oplus a_n {\C}^{c_n}$. 

Applying Theorem \ref{theorem:geometric resolution} $c_i$ times for every D5 brane in $\X(r,c)$, we get
\begin{corollary}
\label{cor: maximal geometric resolution}
%Let $d=\sum_{k=1}^m  r_k=\sum_{k=1}^n  c_k$. 
There exists an $\At$-equivariant embedding $\X(r,c)\hookrightarrow \TsFl(r,d)$.  Moreover, for any fixed point $f\in \X(r,c)^{\At}$ the $\At$-fixed component $F\subseteq \TsFl(r,d)^{\At}$ such that $j(f)\in F$ is the $n$-fold fiber product 
\begin{equation}
    \label{eq:corollary embedding in flags}
    F= \TsFl(1^{c_1}, c_1)\times \dots \times \TsFl(1^{c_n}, c_n).
\end{equation}
\end{corollary}
\begin{definition}
\label{def:maximal resolutions}
    Let $f\in \X(r,c)^{\At}$ be a fixed point corresponding to an element $M\in \BCT(r,c)$. The associated fixed points $F^{\widetilde \At}\subset X^{\widetilde \At}$ are called maximal resolutions of $f$. Similarly, the BCT tables corresponding to the fixed points $F^{\widetilde \At}\subset X^{\widetilde \At}$ are called maximal resolutions of $M$.
\end{definition}
The combinatorial interpretation of this corollary is a straightforward generalization of the points (i)-(iii) above. Each term in the $n$-fold fiber product \eqref{eq:corollary embedding in flags} corresponds to a column of $M$, and the $c_i!$ fixed points in $\TsFl(1^{c_i}, c_i)$ correspond to all possible ways of ``resolving'' the $i$-th column of $M$ with $c_i$ columns with exactly one nonzero entry. Hence, altogether, the maximal resolutions $\tilde f \in F^{\widetilde \At}$ correspond to the $\prod_{i=1}^{n} c_i!$ ways of ``maximally resolving'' $M$. The resulting matrices $\widetilde M\in \BCT(r,1^d)$ are exactly the fixed points in $\TsFl(r,d)$.

\begin{example}
    Consider the bow variety $\X(r,c)=\ttt{\fs1\fs2\fs4\bs2\bs}$. Then, $r=(1,1,2)$ and $c=(2,2)$, and two fixed points in $\X(r,c)^{\At}$ are described by the BCTs
    \[
M=\begin{pmatrix}
1&0\\
0&1\\
1&1\\
\end{pmatrix}
\qquad 
N=\begin{pmatrix}
0&1\\
1&0\\
1&1\\
\end{pmatrix}.
    \]
Each BCT admits four maximal resolutions:
\begin{align*}
M_1&=\left(
\begin{array}{cc|cc}
1&0&0&0\\
0&0&1&0\\
0&1&0&1\\
\end{array}
\right)
\; 
M_2&=\left(
\begin{array}{cc|cc}
0&1&0&0\\
0&0&1&0\\
1&0&0&1\\
\end{array}
\right)
\; 
M_3&=\left(
\begin{array}{cc|cc}
1&0&0&0\\
0&0&0&1\\
0&1&1&0\\
\end{array}
\right)
\; 
M_4&=\left(
\begin{array}{cc|cc}
0&1&0&0\\
0&0&0&1\\
1&0&1&0\\
\end{array}
\right)
\\
N_1&=\left(
\begin{array}{cc|cc}
0&1&1&0\\
1&0&0&0\\
0&0&0&1\\
\end{array}
\right)
\; 
N_2&=\left(
\begin{array}{cc|cc}
1&0&1&0\\
0&1&0&0\\
0&0&0&1\\
\end{array}
\right)
\; 
N_3&=\left(
\begin{array}{cc|cc}
0&1&0&1\\
1&0&0&0\\
0&0&1&0\\
\end{array}
\right)
\; 
N_4&=\left(
\begin{array}{cc|cc}
1&0&0&1\\
0&1&0&0\\
0&0&1&0\\
\end{array}
\right)
\end{align*}
\end{example}

The case when $X$ has exactly two D5 branes, and hence its BCTs have two rows, will be the most relevant for us. As a consequence, we invite the reader to keep in mind the example above.

\subsection{Second proof of Theorem \ref{thm:main}}
\label{section: second proof main thm}

Theorem \ref{thm:classification} provides an explicit description of all invariant curves in a bow variety  $\X(r,c)$ in terms of certain combinatorial operations for BCT tables and, hence, opens the way to our first proof of Theorem \ref{thm:main}. However, the results introduced in the previous section suggest an alternative approach to its proof. Indeed, instead of classifying all the invariant curves, we can exploit Corollary \ref{cor: maximal geometric resolution} to control them in terms of the invariant curves in a partial flag variety $\Fl(r,d)$. As we will see, this will be the key idea behind our second proof of the main theorem.

We begin with two technical lemmas. The first one characterizes the secondary Bruhat order for BCT tables with two columns, and the second one provides a geometric analog of Lemma~\ref{lma: Brhuaht order and resolutions}.

%These are, respectively, Lemma \ref{lm:attracting sets different torus actions projective variety} and Lemma \ref{lma: 2-columns characterization}.

\begin{lemma}
\label{lma: 2-columns characterization}
    Let $M,M'$ be two BCTs with $n$ rows and exactly two columns. Then $M'\leqsb M $ if and only if $M'=M+Z$ where $Z\in \Mat(n,2)$ such that
    \begin{enumerate}
        \item[(i)] each row of $Z$ is either $(0,0)$, $(1,-1)$ or $(-1,1)$.
        \item[(ii)] $\# \lbrace (1,-1) \text{ among the first $k$ rows of $Z$} \rbrace \geq \# \lbrace (-1,1) \text{ among the first $k$ rows of $Z$} \rbrace$ for all $k\leq n$.
        \item[(iii)] $\# \lbrace (1,-1) \text{ in $Z$} \rbrace = \# \lbrace (-1,1) \text{ in $Z$} \rbrace$.
    \end{enumerate}
\end{lemma}

We invite the reader to test this result in the situation described in Figure \ref{fig:DualGr(2,4)}.

\begin{proof}

For a given matrix $A\in \Mat(2,2)$, let $A^{ij}\in\Mat(n,2)$ be the matrix all whose rows are $(0,0)$ except from the $i$-th and the $j$-th, which we set to be equal to the first and the second row of $A$, respectively. Note that the matrices $L_2$ and $I_2$ involved in the transition defining the secondary Bruhat order are related by 
\[
L_2+\begin{pmatrix} 1 & -1 \\ -1 & 1 \end{pmatrix} =I_2.
\]
As a consequence, if $M'\leqsb M $ then $M'=M+\sum_{k} + Z$, where
\begin{equation}
\label{eq: lemma 2-columns characterization}
Z = \sum_{k=1}^K \begin{pmatrix} 1 &-1 \\ -1 & 1 \end{pmatrix}^{i_kj_k}
\end{equation} 
for some $i_k, j_k\leq n$ and $K\in \N$. As a consequence, the rows of $Z$ are of the form $n_1(1,-1)+n_2(-1,1)$ for some $n_1,n_2\in \N$. But since the rows of $M$ and $M'$ are either $(0,0)$, $(1,0)$, $(0,1)$ or $(1,1)$, the condition $M'=M+\sum_{k}Z$ forces (i). Conditions (ii) and (iii) are then immediate consequences of equation \eqref{eq: lemma 2-columns characterization}.

Conversely, given a $Z$ satisfying (i)-(iii) such that  $M'=M+Z$, let $K$ be the number of $(1,-1)$-type (equivalently $(-1,1)$-type by (iii)) rows in $Z$. For $k=1,\dots, K$, let $i_k$ (resp. $j_k$) be the column where the $k$-th row of type $(1,-1)$ (resp. of type $(-1,1)$) appears. Then, $Z$ is of the form \eqref{eq: lemma 2-columns characterization} and each term in the summation corresponds to a $L_2\to I_2$ interchange, hence $M'\leqsb M$. 
\end{proof}

\begin{lemma}
\label{lm:attracting sets different torus actions projective variety}
Let $ A=A_1\times A_2$ be a torus acting on a complete variety $X$ and let $F,G$ be two connected components for the action of the subtorus $A_1\subset A$.
Let $\sigma: \Cs\to A$ be a generic cocharacter and denote by $\sigma_1:\Cs\to A_1$ the composition of $\sigma$ with the projection $A\to A_1$. Then $\overline{\Attr_{\sigma_1}(F)}\cap G\neq 0$ iff there exist $A$-fixed points $\tilde f\in F^{A}$, $\tilde g\in G^{A}$ such that $\overline{\Attr_{\sigma}(f)}\cap g\neq 0$.
\end{lemma}
\begin{proof}
The proof is analogous to Lemma \ref{lma:one dimensional orbits}. Assume first that $\overline{\Att{\sigma}(F)}\cap G\neq 0$, or, equivalently, that there exists a point $x\in X$ and an orbit $\Cs\ni t\mapsto \sigma_1(t)\cdot x$ such that $\lim_{t\to 0} \sigma_1(t)\cdot x\in F$ and $\lim_{t\to \infty} \sigma_1(t)\cdot x\in G$. The cocharacter $\sigma:\Cs\to A$ splits as 
\[
\begin{tikzcd}
    \Cs\arrow[r, "\sigma"]\arrow[d, , swap, "\Delta"] & A_1\times A_2\\
    \Cs\times \Cs\arrow[ur, swap, "\sigma_1\times \sigma_2"]
\end{tikzcd}
\]
Since $\sigma$ is generic for the action of $A$, the cocharacters $\sigma_1$ and $\sigma_2$ are generic for the actions of $A_1$ and $A_2$, respectively. Consider the action of $\Cs\times \Cs$ on $X$ induced by $\sigma_1\times \sigma_2$ and the map 
\begin{equation}
    \label{map in proof attracting set different torus actions}
    \Gamma: \Pe^1\times \Pe^1\xrightarrow[]{p^{-1}} \overline{\Graph(x)}\to X
\end{equation}
defined via Lemma \ref{lemma fixed  points rank two torus}. The closure of the orbit $\Cs\ni t\mapsto \sigma_1(t)\cdot x$ is contained in the image of \eqref{map in proof attracting set different torus actions}, and in particular the limit points $\lim_{t\to 0} \sigma_1(t)\cdot x$ and $\lim_{t\to \infty} \sigma_1(t)\cdot x$ are contained in the images of $\lbrace 0 \rbrace\times \Pe^1$ and $\lbrace \infty \rbrace\times \Pe^1$ respectively. Since the cocharacter $\sigma_1$ is generic, by the second point of Lemma \ref{lemma fixed  points rank two torus} these limit points are fixed by  $A_1$. Then, by connectedness, it follows that 
\begin{equation}
    \label{inclusion in proof attracting set different torus actions}
    \Gamma(\lbrace 0 \rbrace\times \Pe^1)\subseteq F\qquad \Gamma(\lbrace \infty \rbrace\times \Pe^1) \subseteq G.
\end{equation}
Now consider the orbit $\sigma(t)\cdot x$. It can be seen as the image of the composition 
\[
\Cs\xrightarrow[]{\Delta} \Cs\times \Cs\to \Pe^1\times \Pe^1\xrightarrow[]{p^{-1}} \overline{\Graph(x)}\to X
\]
Its limit point $\lim_{t\to 0}\sigma(t)\cdot x$ coincides with $\Gamma(0\times 0)$, hence \eqref{inclusion in proof attracting set different torus actions} implies that $f:=\lim_{t\to 0}\sigma(t)\cdot x\in F$. By the third point of Lemma \ref{lemma fixed  points rank two torus}, the point $f$ is fixed by the action of $\Cs\times \Cs$ induced by $\sigma= \sigma_1\times \sigma_2$. Since the cocharacters are generic, the limit point $f$ is even fixed by the whole $ A=A_1\times A_2$. Analogously, one can argue that $g:=\lim_{t\to \infty}\sigma(t)\cdot x$ is equal to $\Gamma(\infty\times \infty)$, belongs to $G$, and is $ A$-fixed. Overall, we have shown that $\lim_{t\to 0}\sigma(t)\cdot x=f\in F^A$ and $\lim_{t\to \infty}\sigma(t)\cdot x=g\in G^A$, which implies that $\overline{\Att\sigma(f)}\cap g\neq 0$.

The opposite direction is analogous. In this case one assumes the existence of an orbit $\sigma(t)\cdot x$ with limit points at $f$ and $g$ and shows that the orbit $\sigma_1(t)\cdot x$ connects $F$ to $G$. Details are left to the reader. As before, the closure of $\sigma(t)\cdot x$ corresponds to the diagonal $\Delta: \Pe^1\to \Pe^1\times \Pe^1$, while $\sigma_1(t)\cdot x$ corresponds to $\Pe^1\times \lbrace 1\rbrace \hookrightarrow\Pe^1\times \Cs\subset \Pe^1\times \Pe^1$.
\end{proof}

We are now ready to give an alternative proof of our main theorem.

\begin{proof}[Second proof of Theorem \ref{thm:main}]
    The $\lbrace \text{combinatorial}\rbrace \implies \lbrace\text{geometric}\rbrace$ direction is the easiest part of the proof. Indeed, an attracting invariant curve connecting two fixed points such that $M\lessdot_{\hat{B}}M'$ can be explicitly constructed following \cite[Section 3.4.1]{Shou}. Therefore, we only prove the opposite direction.
    
    For convenience, we denote by $M_X$ a BCT table defining a fixed point in the bow variety $X$. It suffices to show that if there exists an invariant curve $\gamma: \Pe^1\to X$ with attractive tangent weight $\chi\in \text{Char}(\At)$ at $0$ and such that $\gamma(0)=M_X$ and $\gamma(\infty)=M_X'$, then $M'_X<M_X$.
    
    First, we reduce ourselves to the case where there are only two D5 branes. This corresponds to the case where the BCT table has only two columns. By \cite[Section 4.6]{RimanyiShou}, the tangent weight $\chi$ of $\gamma$ at $0$ must be of the form $a_i/a_j$ for some $i\neq j\in \lbrace 1,\dots , n\rbrace$. Let $B\subset \At$ be the sub-torus given by the equation $a_i=a_j$. Then, the curve $\gamma$ is contained in some connected component $Y\subset X^B$. Moreover, its endpoints $M_X$ and $M'_X$ correspond to two $\At/B$-fixed points in $Y$. By \cite[Theorem 3.1]{BottaRimanyi}, every fixed component in $X^B$ is isomorphic to a bow variety with exactly two branes (these two branes correspond to the $i$-th and $j$-th D5 branes in $X$). Let $M_Y$ and $M_Y'$ be the BCT tables describing the fixed points $\gamma(0), \gamma(\infty)\in Y^{\At/B}$. The matrix $M_Y$ (resp. $M'_Y$) is simply obtained from $M_X$ (resp. $M'_X$) by erasing all but the $i$-th and $j$-th columns. In particular, if we show that $M'_Y<_{\hat B} M_Y$, then $M'_X<_{\hat B} M_X$ follows. Moreover, since by construction $\gamma$ is contained in $Y\subset X$, it follows that $M'_Y < M_Y$. Overall, we have shown that 
     \[
     M'_X <_G M_X \implies M'_Y <_G M_Y\qquad M'_Y\leqsb M_Y \implies M'_X\leqsb M_X.
     \]
     Hence, it suffices to prove that $M'_Y <_G M_Y \implies M'_Y<_{\hat B} M_Y$. In words, we have reduced ourselves to the case of a bow variety with only two D5 branes (equivalently, to the case of BCTs with only two columns). This concludes the first step of the proof. 
     
     In the second step, we exploit the results of the previous section to constrain the BCT table of a fixed point $M'_Y$ satisfying $M'_Y <_G M_Y$. To ease the notation, we drop the subscripts in the fixed points since there is no more risk of confusion. By Corollary \ref{eq:corollary embedding in flags}, there is an $\At$-equivariant embedding $Y\hookrightarrow \TsFl$, where $\Fl$ is a partial flag variety. Let $F\subseteq \TsFl^{\At}$ (resp. $G\subseteq \TsFl^{\At}$) be the $\At$-fixed component containing $M$ (resp. $M'$). Let $F_0$ (resp. $G_0$) be the intersection of $F$ (resp. $G$) with the zero section $\Fl\subset \TsFl$. They are both isomorphic to the fiber products of two full flag varieties. Since by hypothesis $M$ and $M'$ are connected by a chain of invariant curves, the same is true for $F$ and $G$ and hence for $F_0$ and $G_0$ via the projection $\TsFl\to \Fl$. Therefore, we deduce that $\overline{\Attr_{\sigma}(F_0)} \cap G_0\neq \emptyset$. Let $\widetilde \At \supset \At$ be the maximal torus acting on $\Fl$ from the framing. 
     
     Choose a splitting $\widetilde\At=\At\times \widetilde{\At}/\At $ and a cocharacter $\lambda_{\tilde\sigma}:\Cs\to \widetilde{\At}$ such that its projection to $\At$ is $\lambda_\sigma:\Cs\to \At$. By Lemma \ref{lm:attracting sets different torus actions projective variety} below, it follows that that there exist $\widetilde \At$-fixed points $\widetilde M,\widetilde M'\in \Fl^{\widetilde \At}$ such that $\overline{\Attr_{\tilde\sigma}(\widetilde M)}\cap \widetilde M'\neq \emptyset$. Notice that $\widetilde M$ (resp. $\widetilde M'$) is a maximal resolution of $M$ (resp.$M'$) in the sense of Definition \ref{def:maximal resolutions}. Since  for flag varieties either combinatorial Bruhat orders coincide with the geometric order—see Remark \ref{GenoaCFC}—we deduce that $\widetilde M'$ is obtained from $\widetilde M$ by a sequence of $L_2 \to I_2$ interchanges.
     If one of these interchanges occurs among the columns resolving a single column of $\widetilde M$, they leave $M$ unchanged. If instead they involve columns of $\widetilde M$ resolving both columns of $M$, they have the effect of adding to $M$ a matrix of the form
     \[
     \begin{pmatrix}
         1 &-1 \\
         -1 & 1
     \end{pmatrix}.
     \]
     Hence, we deduce that $M'=M+Z$, where
     \[
     Z = \sum_{k} \begin{pmatrix} 1 &-1 \\ -1 & 1 \end{pmatrix}^{i_kj_k}
     \]
    for some indices $1\leq i_k, j_k\leq m$.
    %Here, given a matrix $A\in \Mat(2,2)$, we denote by $A^{i,j}\in\Mat(n,2)$ the matrix all whose rows are $(0,0)$ except from the $i$-th and the $j$-th, which we set to be equal to the first and the second row of $A$, respectively. 
    But then the same argument used in the proof of Lemma \ref{lma: 2-columns characterization} shows that the assumptions of the latter are satisfied. Hence the $M<_{\hat B} M'$ and the proof follows.
\end{proof}

%The following lemma can be thought of as the geometric counterpart of Lemma \ref{lma: Brhuaht order and resolutions}.

%%%%%%%%%%%%%%%%%%%%%%%%%%%%%%%%%%
%%%%%%%%%%%%%%%%%%%%%%%%%%%%%%%%%%

\appendix

\section{}

In this appendix, we record some results about orbits of torus actions on complete algebraic varieties. In the first lemma, which is classical, we recall the characterization of $\Cs$-oribits  closures in terms of morphisms from $\Pe^1$. Then, we extend the result to rank two tori. Although this second lemma is also probably well known, we could not find a reference for it. Hence, we provide a proof.

\begin{lemma}[Lemma 2.4.1 in \cite{Chriss_book}] 
\label{lemma curves in varieties}
Let X be a complete variety and let $\Cs$ be one dimensional torus acting on $X$. Fix $x\in X$ and let $\overline{\Graph(x)}$ be the closure in $\Pe^1 \times X$ of $\Graph(x)= \{ (z, z\cdot x )\in \Cs\times X \; | \;z\in \Cs\}$.
Then
\begin{enumerate}
    \item The projection $p: \overline{\Graph(x)}\to \Pe^1$ is an isomoprhism.
    \item The image of the points $0,\infty\in \Pe^1$ under the composition $\Pe^1\xrightarrow[]{p^{-1}} \overline{\Graph(x)}\to X$  are fixed points for the $\Cs$-action.
\end{enumerate}

\end{lemma}

%\begin{proof}
%By construction, $\overline{\Graph(x)}$ is a one dimensional projective variety. Since $\Graph(x)\cong \Cs$, the map $\overline{\Graph(x)}\to \Pe^1$ is a birational isomorphism. Let Y be the normalization of $\overline{\Graph(x)}$. It is smooth and projective, and the composition $Y\to \overline{\Graph(x)} \to \Pe^1$ is birational. This implies that $Y$ has genus zero, hence $Y\cong \Pe^1$. Since any birational isomorphism of $\Pe^1$ is a genuine isomoprhism, we conclude that also $\Graph(x)\to \Pe^1$ is an isomorphism. 

%We now prove (2). By (1), the closure of the orbit $\Cs\cdot x$ contains exactly two points, corresponding to $0$ and $\infty$ in $\Pe^1$. Since the boundary of the closure of any orbit is a disjoint union of orbits of smaller dimensions, we conclude that these two points are fixed by $\Cs$.
    
%\end{proof}

With this well known result in mind, we extend the previous result to rank two tori. In principle, we could admit an arbitrary rank, but this level of generality suffices for our scopes.
\begin{lemma}
\label{lemma fixed  points rank two torus}
    Let X be a complete variety and let $\Cs\times \Cs$ act on $X$ via $m: \Cs\times \Cs\times X\to X$. Fix $x\in X$ and let $\overline{\Graph(x)}$ be the closure in $\Pe^1\times \Pe^1 \times X$ of
    \[
    \{ (z,w,m(z,w,x))\in \Cs\times \Cs\times X \; |\; (z,w)\in \Cs\times \Cs\}.
    \]
    Then
    \begin{enumerate}
        \item The projection $p: \overline{\Graph(x)}\to \Pe^1\times \Pe^1$ is an isomorphism.
        \item The images of the curves $\lbrace 0\rbrace \times \Pe^1$ and $\lbrace \infty \rbrace \times \Pe^1$ under the composition $\Pe^1\times \Pe^1\xrightarrow[]{p^{-1}} \overline{\Graph(x)}\to X$ are fixed by the action of the one dimensional torus $\Cs\times \lbrace 1\rbrace \subset \Cs\times \Cs $. Similarly, the images of $\Pe^1\times \lbrace 0\rbrace $ and $\Pe^1\times \lbrace \infty \rbrace$ are fixed under the action of the one dimensional torus $\lbrace 1\rbrace \times \Cs \subset \Cs\times \Cs $. 
        \item The images of the ``corners'' $0\times 0, 0\times \infty, \infty\times 0, \infty\times \infty \in \Pe^1\times \Pe^1$ are fixed under the action of $\Cs\times \Cs$.
    \end{enumerate}
\end{lemma}
\begin{proof}
Since $ \overline{\Graph(x)}$ is connected and $ \Pe^1\times \Pe^1$ is normal, it suffices to show that $\pi: \overline{\Graph(x)}\to \Pe^1\times \Pe^1$ is bijective on $\C$-points. The proof then essentially follows by iteration of Lemma \ref{lemma curves in varieties}. 
    Fix $z\in \Cs \subset \Pe^1$ and consider the following pullback diagram
    \[
    \begin{tikzcd}
        \overline{\Graph(x)}_z \arrow[d, hookrightarrow] \arrow[r, hookrightarrow] & \overline{\Graph(x)} \arrow[d, hookrightarrow]\\
         \lbrace z \rbrace\times \Pe^1\times X \arrow[r, hookrightarrow] & \Pe^1\times \Pe^1 \times X
    \end{tikzcd}
    \]
    By construction, $\overline{\Graph(x)}_z $ is closed in $\lbrace z \rbrace\times \Pe^1\times X$. Indeed, it coincides with the closure of the orbit $\{ (w, m(z,w, x))\in \Cs\times X \; | \; w\in \Cs\}$. By the previous Lemma \ref{lemma curves in varieties}, we deduce that the composition $\overline{\Graph(x)}_z\to \lbrace z \rbrace\times \Pe^1\times X\to  \lbrace z \rbrace\times \Pe^1$ is an isomorphism for all $z\in \Cs$. As a consequence, it follows that the restriction of $p$ on the preimage of $\Cs\times \Pe^1$ is a bijection. Swapping the order of the two $\Cs$, one similarly shows that the restriction of $p$ to the preimage of $\Pe^1\times \Cs$ is also a bijection. Altogether, we deduce that the restriction of $p$ to $\Pe^1\times \Pe^1\setminus \lbrace 0\times 0, 0\times \infty,\infty\times 0, \infty\times \infty\rbrace$ is bijective. 
    Now consider the pullbacks
    \[
    \begin{tikzcd}
        \overline{\Graph(x)}_\Delta \arrow[d, hookrightarrow] \arrow[r, hookrightarrow] & \overline{\Graph(x)} \arrow[d, hookrightarrow]\\
         \Pe^1\times X \arrow[r, hookrightarrow, "\Delta\times id"] & \Pe^1\times \Pe^1 \times X
    \end{tikzcd}
    \qquad 
    \begin{tikzcd}
        \overline{\Graph(x)}_\nabla \arrow[d, hookrightarrow] \arrow[r, hookrightarrow] & \overline{\Graph(x)} \arrow[d, hookrightarrow]\\
         \Pe^1\times X \arrow[r, hookrightarrow, "\nabla\times id"] & \Pe^1\times \Pe^1 \times X
    \end{tikzcd}
    \]
    along the diagonal $\Delta :\Pe^1\to \Pe^1\times \Pe^1$ and the antidiagonal $\nabla: \Pe^1\to \Pe^1\times \Pe^1$ respectively (here $\nabla$ is the composition of $\Delta$ with the inversion $w\to 1/w$ on the second factor). The variety $ \overline{\Graph(x)}_\Delta$ corresponds to the image of the orbit $\{(z,z, m(z,z, x))\in \Cs\times X\; | \;z\in \Cs\}$. Likewise, the variety $ \overline{\Graph(x)}_\nabla$ corresponds to the image of the orbit $\{(z,1/z, m(z,1/z, x))\in \Cs\times X\; | \; z\in \Cs\}$. Again, by Lemma \ref{lemma curves in varieties} it follows that the restrictions of $p$ to the preimages of $\Delta$ and $\nabla$ are bijections. Since these restrictions clearly agree on their common intersection with the restriction of $p$ on the preimage of $\Pe\times \Pe\setminus \lbrace 0\times 0, 0\times \infty,\infty\times 0, \infty\times \infty\rbrace$, we conclude that $p: \overline{\Graph(x)}\to \Pe^1\times \Pe^1$ is bijective, as desired. This completes the proof of the first point. 
    
    We now prove the second point. By construction, the image of $\lbrace 0\rbrace \times \Cs$ is fixed by $\Cs\times \lbrace 1 \rbrace $. Since fixed loci are closed, also the preimage of $\lbrace 0\rbrace \times \Pe^1$ is closed. This proves the claim for the first curve. The other cases are analog.

    As for the third point, simply notice that by point (2) each of the ``diagonal'' points $0\times 0, 0\times \infty, \infty\times 0, \infty\times \infty \in \Pe^1\times \Pe^1$ lie at the intersection of two curves. One of these curves is fixed by $\Cs\times \lbrace 1 \rbrace $, while the other is fixed by $\lbrace 1 \rbrace\times \Cs $, hence their intersection if fixed by the whole torus $\Cs\times\Cs$.

\end{proof}

%\begin{lemma}
%    Let $A$ be a torus acting on a complete variety $X$ and let $\lambda: \Cs\to A$ be generic character. Then two fixed points $f,g\in X^{A}=X^{\C}$ are connected by an attractive $\Cs$-orbit iff they can be connected via a chain of one dimensional attractive $A$-orbits.
%\end{lemma}
%\begin{proof}
%    We argue by induction on the rank of the torus $A$. If $\rk{A}=1$, then the statement is trivial. Otherwise, let $\chi$ be an $A$-weight in the tangent space of $A$ and set $A=\text{Ker}(\chi)$. 
%\end{proof}

\bibliographystyle{amsplain}
\bibliography{BruhatRefs}

\end{document}